\makeatletter
\@namedef{subjclassname@2020}{\textup{2020} Mathematics Subject Classification}
\makeatother

\documentclass{amsart}
\usepackage{latexsym,rotate,eucal,cite}
\usepackage{amsmath,amsthm,amssymb,amsxtra}
\usepackage[pagewise]{lineno}
\usepackage{cite}
\usepackage{amsfonts}
\usepackage{latexsym, amssymb, amsmath,amssymb, amscd, amsfonts, epsfig, graphicx, colordvi}
\usepackage{ifpdf}
\usepackage{graphicx, epstopdf}
\usepackage{xcolor}
\usepackage{dsfont}
\usepackage[dvipsnames]{pstricks}
\usepackage{cases}
\newtheorem{theorem}{Theorem}[section]
\newtheorem{lemma}[theorem]{Lemma}
\newtheorem{proposition}[theorem]{Proposition}
\newtheorem{corollary}[theorem]{Corollary}

\theoremstyle{definition}
\newtheorem{definition}[theorem]{Definition}
\newtheorem{example}[theorem]{Example}

\newtheorem{remark}[theorem]{Remark}

\numberwithin{equation}{section}

\allowdisplaybreaks[4]

\makeatletter
\def\ExtendSymbol#1#2#3#4#5{\ext@arrow 0099{\arrowfill@#1#2#3}{#4}{#5}}
\makeatother

\begin{document}

\title{On $f$-polyharmonic maps between Riemannian manifolds}

\author{Xin Zhan}
\address{School of Mathematics and Statistics, Suzhou University of Technology,
SuZhou 215500, P. R. China} \email{zhanxin\_math@163.com}





\begin{abstract}
This paper is devoted to a general study of $f$-polyharmonic maps of order $k$ (or $f$-$k$-harmonic maps), defined as critical points of the weighted $k$-energy functional
\[
E_{f,k}(\phi)=\frac{1}{2}\int_\Omega f  |\overline{\Delta}^{k/2}\phi|^2 dv_g.
\]
This framework provides a unifying perspective that extends previous theories including $f$-harmonic maps ($k=1$), biharmonic and $f$-biharmonic maps ($k=2$), and polyharmonic maps ($k\ge 3$ with constant $f$), with the classical harmonic maps recovered as the special case $k=1$ by setting $f\equiv \mathrm{const}$. We derive the Euler--Lagrange equation for general $f$-polyharmonic maps. As concrete applications, we classify $f$-$k$-harmonic curves with positive constant geodesic curvature in a space form $N^2(C)$ for $k=3,4$. Several explicit constructions of proper $f$-polyharmonic functions and maps are also provided, and a Liouville-type theorem is proved: every $f$-polyharmonic function on a closed Riemannian manifold is constant.
\end{abstract}

\maketitle \markboth{X. Zhan} {On $f$-polyharmonic maps between Riemannian Manifolds}

\section{Introduction}

The theory of harmonic maps is widely regarded as one of the most productive and transformative developments in contemporary differential geometry and geometric analysis. A smooth map $\phi: M \to N$ between Riemannian manifolds $(M^m, g)$ and $(N^n, h)$ is said to be \textit{harmonic} if it is a critical point of the energy functional
\begin{equation}\label{energy}
E(\phi) = \int_\Omega e(\phi) \, dv_g
\end{equation}
for any compact domain $\Omega$ of $M$, where $e(\phi):=\frac 12 |d\phi|^2 = \frac 12\text{tr}_g \phi^*h$ denotes the energy density and $dv_g$ is the volume element on $M$. The origins of harmonic map theory can be traced back to the classical work of Eells and Sampson \cite{EellsSampson1964}, which established the existence of harmonic maps from compact manifolds to non-positively curved targets via the heat flow method. This seminal paper inaugurated a rich research area that has since connected with numerous branches of mathematics, including partial differential equations, topology, complex geometry, and theoretical physics. The Euler-Lagrange equation characterizing harmonic maps is the vanishing of the \textit{tension field}
\begin{equation}\label{tension}
\tau(\phi) := \text{tr}_g \nabla d\phi = 0,
\end{equation}
which constitutes a system of second-order semi-linear elliptic partial differential equations. For comprehensive treatments of harmonic map theory, we refer to the foundational texts \cite{EellsLemaire1983, EellsSampson1964}.

As a natural generalization of harmonic maps, Eells and Sampson \cite{EellsSampson1965} introduced the concept of \textit{ES-$k$-harmonic} maps in their influential monograph (see also \cite{EellsLemaire1983,Montaldo2022}). A map $\phi: M \to N$ is called \textit{ES-$k$-harmonic} if it is a critical point of the ES-$k$-energy functional
\begin{equation}\label{ES-k-energy}
E_k^{ES}(\phi) = \frac{1}{2} \int_\Omega |(d+d^*)^k \phi|^2 \, dv_g
\end{equation}
for any compact domain $\Omega$ of $M$, where $d$ and $d^*$ denote the differential and codifferential, respectively.

Another generalization of harmonic maps is given by the so-called \textit{polyharmonic maps of order $k$} (also referred to as \textit{$k$-harmonic maps}), defined as critical points of the $k$-energy functional
\begin{equation}\label{k-energy}
E_k(\phi) = \frac{1}{2} \int_\Omega \bigl| \overline{\Delta}^{k/2} \phi \bigr|^2 \, dv_g
\end{equation}
for any compact domain $\Omega$ of $M$, where $\overline{\Delta} = -\operatorname{tr}_g(\overline{\nabla}^2)$ denotes the rough Laplace operator acting on sections of the pullback bundle $\phi^{-1}TN$. The operator $\overline{\Delta}^{k/2}$ is defined iteratively as follows:
\begin{itemize}
    \item For even $k = 2s$,
    \[
    \overline{\Delta}^{s} \phi := \underbrace{\overline{\Delta} \circ \overline{\Delta} \circ \cdots \circ \overline{\Delta}}_{s \text{ times}} \phi.
    \]
    \item For odd $k = 2s+1$,
    \[
    \overline{\Delta}^{s+\frac 1 2} \phi := \overline{\nabla} \bigl( \overline{\Delta}^{s} \phi \bigr),
    \]
    \end{itemize}
where $\overline{\nabla}$ denotes the induced connection on $\phi^{-1}TN$, and $\tau(\phi) = -\overline{\Delta} \phi$ is the tension field of $\phi$.

For $k = 2, 3$, the functionals $E_k^{ES}(\phi)$ and $E_k(\phi)$ coincide, whereas for $k \geq 4$ they differ \cite{Montaldo2022}. The case $k=2$, known as \textit{biharmonic maps}, has received particular attention and developed into an independent research field. The bienergy functional
\begin{equation}\label{bienergy}
E_2(\phi) = \frac{1}{2} \int_\Omega |\tau(\phi)|^2 \, dv_g
\end{equation}
was first systematically studied by Jiang \cite{Jiang1986}, where he derived the first and second variation formulas of $E_2(\phi)$. The Euler-Lagrange equation for biharmonic maps is given by the fourth-order elliptic equation
\begin{equation}\label{biharmonic-equation}
\tau_2(\phi):=\overline{\Delta} \tau(\phi) + \text{tr}_g R^N\big(d\phi, \tau(\phi)\big) d\phi = 0,
\end{equation}
where $R^N$ denotes the curvature tensor of the target manifold $N$. Moreover, the first and second variation formulas for \eqref{k-energy} were derived by Wang \cite{Wang1989} and Maeta \cite{Maeta-2012Osaka}, respectively.

The biharmonic equation has been extensively studied in various ambient spaces.
For an isometric immersion $\phi: M^n \to \mathbb{R}^{n+1}$ of a hypersurface into Euclidean space, it reduces to $\overline{\Delta} \mathbf{H} = 0$, where $\mathbf{H}$ denotes the mean curvature vector field. This led to the famous \textit{Chen's conjecture} \cite{Chen1991}, which asserts that the only biharmonic submanifolds in Euclidean spaces are the minimal ones. This conjecture has stimulated extensive research and has been verified in numerous special cases:
for surfaces in $\mathbb{R}^3$ \cite{Chen1991,Jiang1987};
for hypersurfaces in $\mathbb{R}^4$ \cite{Defever1998,Hasanis1995};
for hypersurfaces with at most two distinct principal curvatures \cite{Dimitric1992};
for those with at most three distinct principal curvatures \cite{Fu2015};
and more recently for hypersurfaces in $\mathbb{R}^5$ \cite{fu-hong-zhan2021} and $\mathbb{R}^6$ \cite{fu-hong-zhan2023}.
However, the general case remains open despite significant efforts.
For a comprehensive overview of the biharmonic theory and its developments, we refer the interested reader to the survey articles \cite{Chen2014-survey,Chen2026-biconservative-survey,Fetcu-Oniciuc2022-survey,Fu-Yang-Zhan2022-survey,Ou2016-survey} and to the monograph \cite{OuChen2020}.

Parallel to these developments, a weighted generalization of harmonic map theory emerged through the study of \emph{$f$-harmonic maps} \cite{EellsLemaire1978,Lichnerowicz1970}. Let $f \in C^\infty(M)$ be a smooth positive function on the domain manifold. The \emph{$f$-energy functional} is defined by
\begin{equation}\label{f-energy}
E_f(\phi) = \frac{1}{2} \int_\Omega f|d\phi|^2 \, dv_g
\end{equation}
for any compact domain $\Omega$ of $M$.
Critical points of $E_f$ are called \emph{$f$-harmonic maps}, a notion that arises naturally in various geometric and physical contexts. The Euler-Lagrange equation for $f$-harmonic maps involves the \emph{$f$-tension field}
\begin{equation}\label{f-tension}
\tau_f(\phi) := f\big[\tau(\phi) + d\phi(\nabla \ln f)\big] = 0,
\end{equation}
which modifies the classical tension field by a gradient term that couples the map with the weight function \cite{Course2024thesis,Ouakkas2010}. The properties of $f$-harmonic maps depend critically on the choice of the function $f$. In physics, $f$ is known as the coupling function and is defined as the continuum limit of the coupling constants between neighboring spins. Thus, $f$-harmonic maps are precisely the stationary solutions of the inhomogeneous Heisenberg spin system \cite{Ou2014-f-harmonic,Li-Wang2006}. The theory of $f$-harmonic maps has been extensively developed, including results on existence and geometric properties \cite{Huang2007,Ou2014-f-harmonic,Rimoldi2013}.

Building upon these foundations, the notion of \textit{$f$-biharmonic maps} was first introduced by Lu \cite{Lu2015} as a natural weighted generalization of biharmonic maps. An $f$-biharmonic map is defined as a critical point of the \emph{$f$-bienergy functional}
\begin{equation}\label{f-bienergy}
E_{f,2}(\phi) = \frac{1}{2} \int_M f \, |\tau(\phi)|^2 \, dv_g.
\end{equation}
The Euler--Lagrange equation for $E_{f,2}(\phi)$ is (see \cite[Proposition 3.6]{Lu2015})
\begin{align}\label{f-biharmonic-equation}
\tau_{f,2}(\phi):=&\overline{\Delta} \big(f\tau(\phi)\big) + \text{tr}_g R^N\big(d\phi, f\tau(\phi)\big) d\phi\\
=&f\tau_2(\phi)-(\Delta f)\tau(\phi)-2\overline{\nabla}_{\nabla f}\tau(\phi) = 0,\nonumber
\end{align}
where $\Delta=\sum\limits_{i=1}^n (e_i e_i-\nabla_{e_i}e_i)$ stands for the Beltrami Laplacian.
Subsequent research has explored various aspects of $f$-biharmonic maps, including their fundamental properties, explicit examples, and classification results for curves and submanifolds \cite{Lu2015, Luo-Ou2019, Ou2014-f-biharmonic, Ou2017}.

Observe that a unified theory for higher-order weighted polyharmonic maps has remained conspicuously absent from the literature. The present paper aims to fill this gap by initiating a rigorous treatment of \textit{$f$-polyharmonic maps of order $k$} (or \textit{$f$-$k$-harmonic maps}). Specifically, we define \textit{$f$-$k$-harmonic maps} as critical points of the \textit{$f$-$k$-energy functional}
\begin{equation}\label{f-k-energy}
E_{f,k}(\phi) = \frac{1}{2} \int_\Omega f \bigl|\overline{\Delta}^{k/2} \phi\bigr|^2 \, dv_g,
\end{equation}
for any compact domain $\Omega$ of $M$.
This unified framework significantly extends previous work: when $f$ is constant, we recover the classical polyharmonic maps studied by Wang \cite{Wang1989,Wang1990,Wang1991} and Maeta \cite{Maeta-2012Proceedings,Maeta-2012Osaka}; when $k=1$, we obtain $f$-harmonic maps; and when $k=2$, we recover $f$-biharmonic maps as special cases.

The remainder of this paper is organized as follows. Section 2 collects the necessary foundational material for $f$-polyharmonic maps between Riemannian manifolds. In Section 3, we derive the Euler--Lagrange equation for general $f$-polyharmonic maps via the first variation of the $f$-$k$-energy functional. In Section~4, we apply the $f$-$k$-harmonic equations to curves, proving the existence of proper $f$-3-harmonic and $f$-4-harmonic curves of positive constant geodesic curvature in $N^2(C)$, whereas no analogous curves arise in the proper $f$-biharmonic setting. Section 5 is devoted to the construction of several concrete radial examples of proper $f$-polyharmonic functions and maps on the punctured Euclidean space $\mathbb{R}^n \setminus \{0\}$. In Section 6, we obtain a new Liouville-type theorem, which states that every $f$-polyharmonic function on a closed Riemannian manifold is constant, whereas on the Euclidean line we construct a large family of proper $f$-$k$-harmonic functions.

\section{Some basic notations for $f$-polyharmonic maps}

In~\cite{Lu2015}, Lu introduced $f$-biharmonic maps between Riemannian manifolds and established the corresponding Euler-Lagrange equation. Motivated by this, we propose a natural generalization of $f$-polyharmonic maps. We begin by setting up some basic notations for $f$-polyharmonic maps.

Let $(M^m,g)$ and $(N^n,h)$ be Riemannian manifolds of dimensions $m$ and $n$, respectively. Denote by $\nabla$ and $\overline{\nabla}$ their Levi-Civita connections, and let $\{e_i\}_{i=1}^m$ be a local orthonormal frame on $M$. The curvature tensor $R^N$ of $N$ is defined by
$$R^N(X,Y)Z = \overline{\nabla}_X \overline{\nabla}_Y Z - \overline{\nabla}_Y \overline{\nabla}_X Z - \overline{\nabla}_{[X,Y]} Z.$$

For a smooth map $\phi:M\to N$, let $\phi^{-1}TN$ be the pullback bundle with induced connection $\overline{\nabla}^\phi$. The \emph{rough Laplacian} on sections of $\phi^{-1}TN$ is defined by
\[
\overline{\Delta}^\phi = -\sum_{i=1}^m \bigl( \overline{\nabla}^\phi_{e_i}\overline{\nabla}^\phi_{e_i} - \overline{\nabla}^\phi_{\nabla_{e_i}e_i}\bigr).
\]
When no confusion arises, we denote $\overline{\nabla}^{\phi}$ and $\overline{\Delta}^{\phi}$ simply by $\overline{\nabla}$ and $\overline{\Delta}$, respectively.

Now take a positive function $f\in C^\infty(M)$. For any integer $k\ge 1$, we define a family of \emph{$f$-polyenergy functionals} $\{E_{f,k}(\phi)\}_{k\ge 1}$ (see Definition~\ref{def:fkenergy}) that unify and extend the standard $k$-energy $E_k(\phi)$, as well as the $f$-energy $E_f(\phi)$ and the $f$-bienergy $E_{f,2}(\phi)$.

\begin{definition}\label{def:fkenergy}
Let $\phi: M^m \rightarrow N^n$ be a smooth map between two Riemannian manifolds $(M^m, g)$ and $(N^n, h)$. For $s \geq 1$, the \emph{$f$-$2s$-energy functional} of $\phi$ is defined by
\begin{align}\label{E-2s}
  E_{f,2s}(\phi)
  &= \frac{1}{2} \int_\Omega f \Bigl\langle \underbrace{(d^*d)\cdots (d^*d)}_{s\text{ times}}\phi,
                                          \underbrace{(d^*d)\cdots (d^*d)}_{s\text{ times}}\phi \Bigr\rangle \, dv_g \\
  &= \frac{1}{2} \int_\Omega f \bigl\langle \overline{\Delta}^{\,s-1}\tau(\phi),
                                          \overline{\Delta}^{\,s-1}\tau(\phi) \bigr\rangle \, dv_g \nonumber
\end{align}
for any compact domain $\Omega$ of $M$. For $s \geq 0$, the \emph{$f$-$(2s+1)$-energy functional} is defined by
\begin{align}\label{E-2s+1}
  E_{f,2s+1}(\phi)
  &= \frac{1}{2} \int_\Omega f \Bigl\langle d\underbrace{(d^*d)\cdots (d^*d)}_{s\text{ times}}\phi,
                                          d\underbrace{(d^*d)\cdots (d^*d)}_{s\text{ times}}\phi \Bigr\rangle \, dv_g \\
  &= \frac{1}{2} \int_\Omega f \bigl\langle \overline{\nabla}\, \overline{\Delta}^{\,s-1}\tau(\phi),
                                          \overline{\nabla}\,\overline{\Delta}^{\,s-1}\tau(\phi) \bigr\rangle \, dv_g \nonumber
\end{align}
for any compact domain $\Omega$ of $M$. Here $\overline{\Delta}^0=id$, $\overline{\Delta}^{-1}\tau(\phi) = - \phi$ and $\overline{\nabla}_X \phi = d\phi(X)$.
\end{definition}

A smooth map $\phi: M \to N$ is called \emph{$f$-polyharmonic of order $k$} (or \emph{$f$-$k$-harmonic}) if it is a critical point of $E_{f,k}$ with respect to all compactly supported variations.
When $\phi: M \to N$ is an isometric immersion, $M$ is called an \emph{$f$-polyharmonic submanifold of order $k$} (or \emph{$f$-$k$-harmonic submanifold}).
An $f$-$k$-harmonic map or submanifold is called \emph{proper} if it is not $k$-harmonic.

Throughout this paper, we adopt the following conventions unless otherwise stated.
\begin{itemize}
    \item $M$ and $N$ are smooth, connected Riemannian manifolds of dimensions $m$ and $n$, respectively;
    \item Indices $i,j,\dots$ run from $1$ to $m = \dim M$;
    \item The index $l$ runs from $1$ to $k-1$ (whenever $k\ge 2$);
    \item The Einstein summation convention is used unless otherwise stated.
\end{itemize}
These basic notations will be essential for the derivation of the Euler-Lagrange equations in Section 3 and for the concrete applications in Sections 4, 5, and 6.

\section{The first variational formula for $E_{f,k}(\phi)$}

In this section, we derive the first variational formula for the $f$-$k$-energy functional $E_{f,k}(\phi)$.

Let $\phi: (M^m,g) \to (N^n,h)$ be a smooth map between Riemannian manifolds.
For $\varepsilon>0$ sufficiently small, let $\{\phi_t\}_{t\in(-\varepsilon,\varepsilon)}$ be a smooth compactly supported variation of $\phi$, and denote by $\Phi: M \times (-\varepsilon,\varepsilon) \to N$ the map given by $\Phi(p,t) = \phi_t(p)$, with $\Phi(p,0) = \phi(p)$.
The corresponding variational vector field $V$ is a section of $\phi^{-1}TN$ given by
\[
V=\left.\frac{d}{dt}\right|_{t=0}\phi_t\in\Gamma(\phi^{-1}TN).
\]

For convenience, we shall abuse notation and denote by $\nabla$ and $\overline{\nabla}$ the Levi-Civita connections on the tangent bundle $T(M^m\times(-\varepsilon,\varepsilon))$ and the pullback bundle $\Phi^{-1}TN$, respectively. Similarly, we denote by $\overline{\Delta}$ the rough Laplacian on $\Phi^{-1}TN$.
The derivation of the first variational formula for $f$-$k$-harmonic maps relies on the following lemmas.
\begin{lemma}\label{lem-1} {\rm (c.f. \cite{Maeta-2012Proceedings,Maeta-2012Osaka,Wang1989})}
\begin{align}
  &\overline{\nabla}_{\frac{\partial}{\partial t}}\, \overline{\Delta}^{\,s-1}\tau(\phi_t)\Big|_{t=0} \label{L-1-1} \\
  ={}& -\overline{\Delta}^{\,s}V
     + \sum_{j=1}^{m}\overline{\Delta}^{\,s-1}R^N\big( V , d\phi(e_j) \big)d\phi(e_j)  \nonumber \\
  &   + \sum_{l=1}^{s-1}\sum_{j=1}^{m} \overline{\Delta}^{\,l-1}\Big\{
        R^N\big(V, d\phi(\nabla_{e_j}e_j)\big)\overline{\Delta}^{\,s-1-l}\tau(\phi) \nonumber \\
  &\quad   - \overline{\nabla}_{e_j}R^N\big( V , d\phi(e_j) \big)\overline{\Delta}^{\,s-1-l}\tau(\phi)  \nonumber \\
  &\quad   - R^N\big( V , d\phi(e_j) \big)\overline{\nabla}_{e_j}\, \overline{\Delta}^{\,s-1-l}\tau(\phi)
     \Big\} , \nonumber\\[6pt]
  &\overline{\nabla}_{\frac{\partial}{\partial t}} \overline{\nabla}_{e_i} \overline{\Delta}^{\,s-1} \tau(\phi_t) \Big|_{t=0} \label{L-1-2}\\
  ={}& -\overline{\nabla}_{e_i} \overline{\Delta}^{\,s} V
     + \sum_{j=1}^{m}\overline{\nabla}_{e_i} \overline{\Delta}^{\,s-1} R^N (V, d\phi(e_j)) d\phi(e_j) \nonumber\\
  &+ \sum_{l=1}^{s-1} \sum_{j=1}^{m}\overline{\nabla}_{e_i} \overline{\Delta}^{\,l-1} \Big\{
         R^N (V, d\phi(\nabla_{e_j} e_j)) \overline{\Delta}^{\,s-1-l} \tau(\phi)\nonumber\\
  &\quad -\overline{\nabla}_{e_j} R^N (V, d\phi(e_j)) \overline{\Delta}^{\,s-1-l} \tau(\phi) \nonumber\\
  &\quad - R^N (V, d\phi(e_j)) \overline{\nabla}_{e_j} \overline{\Delta}^{\,s-1-l} \tau(\phi)
        \Big\} \nonumber\\
  & +  R^N (V, d\phi(e_i)) \overline{\Delta}^{\,s-1} \tau(\phi).   \nonumber
\end{align}
\end{lemma}
\begin{lemma}\label{lem-2}
Let $X,Y\in \Gamma(\phi^{-1}TN)$ and $p$ be a positive integer. Then for any compact domain $\Omega$ of $M$, we have
\begin{align}
  \int_{\Omega} \bigl\langle \overline{\nabla}X, \overline{\nabla}Y \bigr\rangle \, dv_g
  &= \int_{\Omega} \bigl\langle \overline{\Delta}X, Y \bigr\rangle \, dv_g
   = \int_{\Omega} \bigl\langle X, \overline{\Delta}Y \bigr\rangle \, dv_g, \label{L-2-1} \\[4pt]
  \int_{\Omega} \bigl\langle \overline{\Delta}^{\,p} X, Y \bigr\rangle \, dv_g
  &= \int_{\Omega} \bigl\langle X, \overline{\Delta}^{\,p} Y \bigr\rangle \, dv_g, \label{L-2-2} \\[4pt]
  \int_{\Omega} \bigl\langle \overline{\nabla}X, f\overline{\nabla}Y \bigr\rangle \, dv_g
  & = \int_{\Omega} \bigl\langle X, f\overline{\Delta}Y-\overline{\nabla}_{\nabla f}Y \bigr\rangle \, dv_g, \label{L-2-3}\\[4pt]
  \int_{\Omega} \bigl\langle \overline{\nabla}X, fd \phi \bigr\rangle \, dv_g
  & = -\int_{\Omega} \bigl\langle X, f\tau(\phi)+d\phi(\nabla f) \bigr\rangle \, dv_g. \label{L-2-4}
\end{align}
\end{lemma}
\begin{proof}
Proofs of formulas \eqref{L-2-1} and \eqref{L-2-2} are given in \cite{Wang1989}. We now derive formula \eqref{L-2-3} and \eqref{L-2-4}. Consider the vector fields
\[
\omega = \bigl\langle X,\, f\overline{\nabla}_{e_j} Y \bigr\rangle e_j,\quad
\theta = \bigl\langle X,\, fd \phi(e_j) \bigr\rangle e_j.
\]
It is easily to see that
\begin{align*}
  \operatorname{div} \omega
  &= \bigl\langle \overline{\nabla}_{e_j} X,\, f\overline{\nabla}_{e_j} Y \bigr\rangle
     + \bigl\langle X,\, e_j(f) \overline{\nabla}_{e_j} Y \bigr\rangle
     - \bigl\langle X,\, f \overline{\Delta} Y \bigr\rangle \\
  &= \bigl\langle \overline{\nabla} X,\, f\overline{\nabla} Y \bigr\rangle
     + \bigl\langle X,\, \overline{\nabla}_{\nabla f} Y \bigr\rangle
     - \bigl\langle X,\, f \overline{\Delta} Y \bigr\rangle, \\[6pt]
  \operatorname{div} \theta
  &= \bigl\langle \overline{\nabla}_{e_j} X,\, f d\phi(e_j) \bigr\rangle
     + \bigl\langle X,\, e_j(f) d\phi(e_j) \bigr\rangle
     + \bigl\langle X,\, f \big(\overline{\nabla}_{e_j}d\phi\big)(e_j) \bigr\rangle \\
  &= \bigl\langle \overline{\nabla} X,\, f d\phi \bigr\rangle
     + \bigl\langle X,\, d\phi(\nabla f) \bigr\rangle
     + \bigl\langle X,\, f \tau(\phi) \bigr\rangle.
\end{align*}
Applying Stokes' theorem then yields \eqref{L-2-3} and \eqref{L-2-4}.
\end{proof}
\begin{lemma}\label{lem-3} {\rm (c.f. \cite{Maeta-2012Proceedings,Maeta-2012Osaka,Wang1989})}
For any $V_1,V_2\in \Gamma(\phi^{-1}TN)$ and any compact domain $\Omega$ of $M$, we have
\begin{align}\label{L-3}
  &\int_{\Omega}\sum_{j=1}^{m} \bigl\langle R^N (V, d\phi(\nabla_{e_j} e_j)) V_1
           -\overline{\nabla}_{e_j} R^N (V, d\phi(e_j)) V_1,
        V_2 \bigr\rangle \, dv_g   \\
  =& \int_{\Omega}\sum_{j=1}^{m} \bigl\langle R^N (V, d\phi(e_j)) V_1,
        \overline{\nabla}_{e_j} V_2 \bigr\rangle \, dv_g. \nonumber
\end{align}
\end{lemma}
\begin{theorem}\label{variation-F}
Let $\phi: (M^m,g) \rightarrow (N^n,h)$ be a smooth map and $k\geq 1$ be an integer.
Then
\begin{equation}\label{First-V}
  \left.\frac{d}{dt}\right|_{t=0} E_{f,k}(\phi_t)
  = -\int_{\Omega} \bigl\langle \tau_{f,k}(\phi),V\bigr\rangle \, dv_g,
\end{equation}
where $\Omega$ is a compact domain of $M$. Furthermore, $\phi$ is an $f$-$k$-harmonic map
if and only if its $f$-$k$-tension field $\tau_{f,k}(\phi)$ vanishes identically.

For $k = 2s$ with $s \geq 1$, we have
\begin{small}
\begin{align}
  \tau_{f,2s}(\phi)
  ={}& \overline{\Delta}^{\,s}\bigl(f \overline{\Delta}^{\,s-1}\tau(\phi)\bigr)
      -\sum_{j=1}^{m} R^N\Bigl( \overline{\Delta}^{\,s-1}\bigl(f \overline{\Delta}^{\,s-1}\tau(\phi)\bigr),
                               d\phi(e_j)\Bigr)d\phi(e_j) \label{f-2s}\\
  &   -\sum_{l=1}^{s-1}\sum_{j=1}^{m}
      \Bigl\{ R^N\Bigl( \overline{\nabla}_{e_j}\, \overline{\Delta}^{\,l-1}\bigl( f \overline{\Delta}^{\,s-1}\tau(\phi)\bigr),
                       \overline{\Delta}^{\,s-1-l}\tau(\phi)\Bigr)d\phi(e_j) \nonumber\\
  &\qquad - R^N\Bigl( \overline{\Delta}^{\,l-1}\bigl( f \overline{\Delta}^{\,s-1}\tau(\phi)\bigr),
                     \overline{\nabla}_{e_j}\, \overline{\Delta}^{\,s-1-l}\tau(\phi)\Bigr)d\phi(e_j)
      \Bigr\}. \nonumber
\end{align}
\end{small}
For $k = 2s+1$ with $s \geq 0$, we have
\begin{small}
\begin{align}
  &\tau_{f,2s+1}(\phi)
  ={} \overline{\Delta}^{\,s}\bigl[ f \overline{\Delta}^{\,s}\tau(\phi)
                     - \overline{\nabla}_{\nabla f} \,\overline{\Delta}^{\,s-1}\tau(\phi) \bigr] \label{f-2s+1}\\
  &   -\sum_{j=1}^{m}\Bigl\{
        fR^N\Bigl(\overline{\nabla}_{e_j}\,\overline{\Delta}^{\,s-1}\tau(\phi),
                  \overline{\Delta}^{\,s-1}\tau(\phi)\Bigr) d\phi(e_j) \nonumber\\
  &\qquad + R^N\Bigl( \overline{\Delta}^{\,s-1}\bigl[ f \overline{\Delta}^{\,s}\tau(\phi)
                     - \overline{\nabla}_{\nabla f} \,\overline{\Delta}^{\,s-1}\tau(\phi) \bigr],
                     d\phi(e_j)\Bigr)d\phi(e_j)
      \Bigr\} \nonumber\\
  &   -\sum_{l=1}^{s-1}\sum_{j=1}^{m}\Bigl\{
        R^N\Bigl( \overline{\nabla}_{e_j}\, \overline{\Delta}^{\,l-1}\bigl[ f \overline{\Delta}^{\,s}\tau(\phi)
                     - \overline{\nabla}_{\nabla f} \,\overline{\Delta}^{\,s-1}\tau(\phi) \bigr],
                  \overline{\Delta}^{\,s-1-l}\tau(\phi)\Bigr)d\phi(e_j) \nonumber\\
  &\qquad - R^N\Bigl( \overline{\Delta}^{\,l-1}\bigl[ f \overline{\Delta}^{\,s}\tau(\phi)
                     - \overline{\nabla}_{\nabla f} \,\overline{\Delta}^{\,s-1}\tau(\phi) \bigr],
                     \overline{\nabla}_{e_j}\, \overline{\Delta}^{\,s-1-l}\tau(\phi)\Bigr)d\phi(e_j)
      \Bigr\}. \nonumber
\end{align}
\end{small}
\end{theorem}
\begin{proof}
By using \eqref{E-2s}, we obtain
\begin{align}
  \left.\frac{d}{dt}\right|_{t=0} E_{f,2s}(\phi_t)
  = & \left.\int_{\Omega} f\Bigl\langle
      \overline{\nabla}_{\frac{\partial}{\partial t}} \overline{\Delta}^{\,s-1}\tau(\phi_t),
      \overline{\Delta}^{\,s-1}\tau(\phi_t)
    \Bigr\rangle \, dv_g \right|_{t=0}\label{P-1}\\
    = & \int_{\Omega} \Bigl\langle
      \left.\overline{\nabla}_{\frac{\partial}{\partial t}} \overline{\Delta}^{\,s-1}\tau(\phi_t) \right|_{t=0},
      f\overline{\Delta}^{\,s-1}\tau(\phi)
    \Bigr\rangle \, dv_g.\nonumber
\end{align}
Applying \eqref{L-1-1} and \eqref{L-2-2} to \eqref{P-1} gives
\begin{align}\label{P-2}
  & \left.\frac{d}{dt}\right|_{t=0} E_{f,2s}(\phi_t)\\
  ={} & \int_{\Omega} \Bigl\langle
      -\overline{\Delta}^{\,s}V
     + \sum_{j=1}^{m}\overline{\Delta}^{\,s-1}R^N\big( V , d\phi(e_j) \big)d\phi(e_j)  \nonumber\\
  &   + \sum_{l=1}^{s-1}\sum_{j=1}^{m} \overline{\Delta}^{\,l-1}\Big\{
        R^N\big(V, d\phi(\nabla_{e_j}e_j)\big)\overline{\Delta}^{\,s-1-l}\tau(\phi) \nonumber \\
  &\quad   - \overline{\nabla}_{e_j}R^N\big( V , d\phi(e_j) \big)\overline{\Delta}^{\,s-1-l}\tau(\phi)  \nonumber \\
  &\quad   - R^N\big( V , d\phi(e_j) \big)\overline{\nabla}_{e_j}\, \overline{\Delta}^{\,s-1-l}\tau(\phi)
     \Big\}  \, , \,   f\overline{\Delta}^{\,s-1}\tau(\phi)    \Bigr\rangle \, dv_g\nonumber\\
  ={} & \int_{\Omega} \Bigl\langle
      -V  \, , \,  \overline{\Delta}^{\,s}\big( f\overline{\Delta}^{\,s-1}\tau(\phi) \big)   \Bigr\rangle \, dv_g\nonumber\\
     & +\int_{\Omega} \sum_{j=1}^{m} \Bigl\langle
       R^N\big( V , d\phi(e_j) \big)d\phi(e_j)   \, , \,   \overline{\Delta}^{\,s-1}\big( f\overline{\Delta}^{\,s-1}\tau(\phi) \big)    \Bigr\rangle \, dv_g\nonumber\\
      &   + \int_{\Omega}\sum_{l=1}^{s-1}\sum_{j=1}^{m}  \Bigl\langle \Big\{
        R^N\big(V, d\phi(\nabla_{e_j}e_j)\big)\overline{\Delta}^{\,s-1-l}\tau(\phi) \nonumber \\
  &\quad   - \overline{\nabla}_{e_j}R^N\big( V , d\phi(e_j) \big)\overline{\Delta}^{\,s-1-l}\tau(\phi)  \nonumber \\
  &\quad   - R^N\big( V , d\phi(e_j) \big)\overline{\nabla}_{e_j}\, \overline{\Delta}^{\,s-1-l}\tau(\phi)
     \Big\}  \, , \,   \overline{\Delta}^{\,l-1}\big( f\overline{\Delta}^{\,s-1}\tau(\phi) \big)    \Bigr\rangle \, dv_g.\nonumber
\end{align}
With \eqref{L-3} and the curvature symmetry $\langle R^N(X,Y)Z,W\rangle = \langle R^N(W,Z)Y,X\rangle$, we find that \eqref{P-2} becomes
\begin{align*}
  & \left.\frac{d}{dt}\right|_{t=0} E_{f,2s}(\phi_t)\\
  ={} & \int_{\Omega} \Bigl\langle
      -V  \, , \,  \overline{\Delta}^{\,s}\big( f\overline{\Delta}^{\,s-1}\tau(\phi) \big)   \Bigr\rangle \, dv_g\\
     & +\int_{\Omega} \sum_{j=1}^{m}\Bigl\langle
       R^N\big( V , d\phi(e_j) \big)d\phi(e_j)   \, , \,   \overline{\Delta}^{\,s-1}\big( f\overline{\Delta}^{\,s-1}\tau(\phi) \big)    \Bigr\rangle \, dv_g\\
      &   +\int_{\Omega} \sum_{l=1}^{s-1}\sum_{j=1}^{m}\Bigl\langle
          R^N\big( V , d\phi(e_j) \big)\overline{\Delta}^{\,s-1-l}\tau(\phi)
       \, , \,  \overline{\nabla}_{e_j}\, \overline{\Delta}^{\,l-1}\big( f\overline{\Delta}^{\,s-1}\tau(\phi) \big)    \Bigr\rangle \, dv_g\\
      &   -\int_{\Omega} \sum_{l=1}^{s-1}\sum_{j=1}^{m}\Bigl\langle
          R^N\big( V , d\phi(e_j) \big)\overline{\nabla}_{e_j}\, \overline{\Delta}^{\,s-1-l}\tau(\phi)
       \, , \,   \overline{\Delta}^{\,l-1}\big( f\overline{\Delta}^{\,s-1}\tau(\phi) \big)    \Bigr\rangle \, dv_g\\
  ={} & \int_{\Omega} \Bigl\langle
      -\overline{\Delta}^{\,s}\big( f\overline{\Delta}^{\,s-1}\tau(\phi) \big)   \, , \,  V  \Bigr\rangle \, dv_g\\
     & +\int_{\Omega} \sum_{j=1}^{m}\Bigl\langle
       R^N\Big( \overline{\Delta}^{\,s-1}\big( f\overline{\Delta}^{\,s-1}\tau(\phi) \big)  , d\phi(e_j) \Big)d\phi(e_j)   \, , \,   V   \Bigr\rangle \, dv_g\\
      &   +\int_{\Omega} \sum_{l=1}^{s-1}\sum_{j=1}^{m}\Bigl\langle
          R^N\Big( \overline{\nabla}_{e_j}\, \overline{\Delta}^{\,l-1}\big( f\overline{\Delta}^{\,s-1}\tau(\phi) \big)  , \overline{\Delta}^{\,s-1-l}\tau(\phi) \Big)d\phi(e_j)
       \, , \,  V   \Bigr\rangle \, dv_g\\
      &   -\int_{\Omega} \sum_{l=1}^{s-1}\sum_{j=1}^{m}\Bigl\langle
          R^N\Big( \overline{\Delta}^{\,l-1}\big( f\overline{\Delta}^{\,s-1}\tau(\phi) \big) ,  \overline{\nabla}_{e_j}\, \overline{\Delta}^{\,s-1-l}\tau(\phi) \Big)d\phi(e_j)
       \, , \,   V    \Bigr\rangle \, dv_g\\
  ={} & -\int_{\Omega} \bigl\langle \tau_{f,2s}(\phi),V\bigr\rangle \, dv_g.
\end{align*}
Similarly, from equations \eqref{E-2s+1}, \eqref{L-1-2} and the identities \eqref{L-2-2}, \eqref{L-2-3}, we get
\begin{align}
  & \left.\frac{d}{dt}\right|_{t=0} E_{f,2s+1}(\phi_t) \label{P-3}\\
      ={} & \int_{\Omega} \sum_{i=1}^m \Bigl\langle
    \left.\overline{\nabla}_{\frac{\partial}{\partial t}} \overline{\nabla}_{e_i} \overline{\Delta}^{\,s-1} \tau(\phi_t) \right|_{t=0},
      f\overline{\nabla}_{e_i}\overline{\Delta}^{\,s-1}\tau(\phi)
    \Bigr\rangle \, dv_g \nonumber\\
    ={}&\int_{\Omega}  \Bigl\langle
    - \overline{\Delta}^{\,s} V
        \,  ,  \,
     f \overline{\Delta}^{\,s}\tau(\phi)
                     - \overline{\nabla}_{\nabla f} \,\overline{\Delta}^{\,s-1}\tau(\phi)
    \Bigr\rangle \, dv_g \nonumber\\
    & + \int_{\Omega} \sum_{j=1}^m \Bigl\langle
     \overline{\Delta}^{\,s-1} R^N (V, d\phi(e_j)) d\phi(e_j)
        \,  ,  \,
      f \overline{\Delta}^{\,s}\tau(\phi)
                     - \overline{\nabla}_{\nabla f} \,\overline{\Delta}^{\,s-1}\tau(\phi)
    \Bigr\rangle \, dv_g \nonumber\\
    &+ \int_{\Omega} \sum_{l=1}^{s-1}\sum_{j=1}^m \Bigl\langle
     \overline{\Delta}^{\,l-1} \Big\{
         R^N (V, d\phi(\nabla_{e_j} e_j)) \overline{\Delta}^{\,s-1-l} \tau(\phi)\nonumber\\
  &\quad -\overline{\nabla}_{e_j} R^N (V, d\phi(e_j)) \overline{\Delta}^{\,s-1-l} \tau(\phi) \nonumber\\
  &\quad - R^N (V, d\phi(e_j)) \overline{\nabla}_{e_j} \overline{\Delta}^{\,s-1-l} \tau(\phi)
        \Big\}
        \,  ,  \,
     f \overline{\Delta}^{\,s}\tau(\phi)
                     - \overline{\nabla}_{\nabla f} \,\overline{\Delta}^{\,s-1}\tau(\phi)
    \Bigr\rangle \, dv_g \nonumber\\
    &+ \int_{\Omega} \sum_{i=1}^m \Bigl\langle
     R^N (V, d\phi(e_i)) \overline{\Delta}^{\,s-1} \tau(\phi)
        \,  ,  \,
      f\overline{\nabla}_{e_i}\overline{\Delta}^{\,s-1}\tau(\phi)
    \Bigr\rangle \, dv_g \nonumber\\
       ={}&\int_{\Omega}  \Bigl\langle
    -  V
        \,  ,  \,
     \overline{\Delta}^{\,s}\bigl[ f \overline{\Delta}^{\,s}\tau(\phi)
                     - \overline{\nabla}_{\nabla f} \,\overline{\Delta}^{\,s-1}\tau(\phi) \bigr]
    \Bigr\rangle \, dv_g \nonumber\\
    & + \int_{\Omega} \sum_{j=1}^m \Bigl\langle
      R^N (V, d\phi(e_j)) d\phi(e_j)
        \,  ,  \,
      \overline{\Delta}^{\,s-1}\bigl[ f \overline{\Delta}^{\,s}\tau(\phi)
                     - \overline{\nabla}_{\nabla f} \,\overline{\Delta}^{\,s-1}\tau(\phi) \bigr]
    \Bigr\rangle \, dv_g \nonumber\\
    &+ \int_{\Omega} \sum_{l=1}^{s-1}\sum_{j=1}^m \Bigl\langle
     \Big\{
         R^N (V, d\phi(\nabla_{e_j} e_j)) \overline{\Delta}^{\,s-1-l} \tau(\phi)\nonumber\\
  &\quad -\overline{\nabla}_{e_j} R^N (V, d\phi(e_j)) \overline{\Delta}^{\,s-1-l} \tau(\phi) \nonumber\\
  &\quad - R^N (V, d\phi(e_j)) \overline{\nabla}_{e_j} \overline{\Delta}^{\,s-1-l} \tau(\phi)
        \Big\}\nonumber\\
    &\quad ,  \,
     \overline{\Delta}^{\,l-1} \bigl[ f \overline{\Delta}^{\,s}\tau(\phi)
                     - \overline{\nabla}_{\nabla f} \,\overline{\Delta}^{\,s-1}\tau(\phi) \bigr]
    \Bigr\rangle \, dv_g \nonumber\\
    &+ \int_{\Omega} \sum_{j=1}^m \Bigl\langle
     R^N (V, d\phi(e_j)) \overline{\Delta}^{\,s-1} \tau(\phi)
        \,  ,  \,
      f\overline{\nabla}_{e_j}\overline{\Delta}^{\,s-1}\tau(\phi)
    \Bigr\rangle \, dv_g. \nonumber
\end{align}
Applying \eqref{L-3} and the symmetry $\langle R^N(X,Y)Z,W\rangle = \langle R^N(W,Z)Y,X\rangle$ of $R^N$ to \eqref{P-3} yields
\begin{small}
\begin{align*}
  & \left.\frac{d}{dt}\right|_{t=0} E_{f,2s+1}(\phi_t)  \\
  ={}&\int_{\Omega}  \Bigl\langle
    -  V   \,  ,  \,
     \overline{\Delta}^{\,s}\bigl[ f \overline{\Delta}^{\,s}\tau(\phi)
                     - \overline{\nabla}_{\nabla f} \,\overline{\Delta}^{\,s-1}\tau(\phi) \bigr]
    \Bigr\rangle \, dv_g \\
    & + \int_{\Omega} \sum_{j=1}^m \Bigl\langle
      R^N (V, d\phi(e_j)) d\phi(e_j)      \,  ,  \,
      \overline{\Delta}^{\,s-1}\bigl[ f \overline{\Delta}^{\,s}\tau(\phi)
                     - \overline{\nabla}_{\nabla f} \,\overline{\Delta}^{\,s-1}\tau(\phi) \bigr]
    \Bigr\rangle \, dv_g \\
    &+ \int_{\Omega} \sum_{l=1}^{s-1}\sum_{j=1}^m \Bigl\langle
      R^N (V, d\phi(e_j)) \overline{\Delta}^{\,s-1-l} \tau(\phi) \, ,  \,
     \overline{\nabla}_{e_j}\overline{\Delta}^{\,l-1} \bigl[ f \overline{\Delta}^{\,s}\tau(\phi)
                     - \overline{\nabla}_{\nabla f} \,\overline{\Delta}^{\,s-1}\tau(\phi) \bigr]
    \Bigr\rangle \, dv_g \\
            &- \int_{\Omega} \sum_{l=1}^{s-1}\sum_{j=1}^m \Bigl\langle
      R^N (V, d\phi(e_j)) \overline{\nabla}_{e_j} \overline{\Delta}^{\,s-1-l} \tau(\phi)\, ,  \,
     \overline{\Delta}^{\,l-1} \bigl[ f \overline{\Delta}^{\,s}\tau(\phi)
                     - \overline{\nabla}_{\nabla f} \,\overline{\Delta}^{\,s-1}\tau(\phi) \bigr]
    \Bigr\rangle \, dv_g \\
        &+ \int_{\Omega} \sum_{j=1}^m \Bigl\langle
     R^N (V, d\phi(e_j)) \overline{\Delta}^{\,s-1} \tau(\phi)
        \,  ,  \,
      f\overline{\nabla}_{e_j}\overline{\Delta}^{\,s-1}\tau(\phi)
    \Bigr\rangle \, dv_g \\
  ={}&\int_{\Omega}  \Bigl\langle
     -\overline{\Delta}^{\,s}\bigl[ f \overline{\Delta}^{\,s}\tau(\phi)
                     - \overline{\nabla}_{\nabla f} \,\overline{\Delta}^{\,s-1}\tau(\phi) \bigr]
     \,  ,  \,    V   \Bigr\rangle \, dv_g \\
    & + \int_{\Omega} \sum_{j=1}^m \Bigl\langle
      R^N \Big( \overline{\Delta}^{\,s-1}\bigl[ f \overline{\Delta}^{\,s}\tau(\phi)
                     - \overline{\nabla}_{\nabla f} \,\overline{\Delta}^{\,s-1}\tau(\phi) \bigr], d\phi(e_j)\Big) d\phi(e_j)
        \,  ,  \,   V   \Bigr\rangle \, dv_g \\
    &+ \int_{\Omega} \sum_{l=1}^{s-1}\sum_{j=1}^m \Bigl\langle
      R^N \Big(\overline{\nabla}_{e_j}\overline{\Delta}^{\,l-1} \bigl[ f \overline{\Delta}^{\,s}\tau(\phi)
                     - \overline{\nabla}_{\nabla f} \,\overline{\Delta}^{\,s-1}\tau(\phi) \bigr], \overline{\Delta}^{\,s-1-l} \tau(\phi)  \Big) d\phi(e_j) \, ,  \,     V   \Bigr\rangle \, dv_g \\
            &- \int_{\Omega} \sum_{l=1}^{s-1}\sum_{j=1}^m \Bigl\langle
      R^N \Big(\overline{\Delta}^{\,l-1} \bigl[ f \overline{\Delta}^{\,s}\tau(\phi)
                     - \overline{\nabla}_{\nabla f} \,\overline{\Delta}^{\,s-1}\tau(\phi) \bigr], \overline{\nabla}_{e_j} \overline{\Delta}^{\,s-1-l} \tau(\phi) \Big) d\phi(e_j)   \, ,  \,    V  \Bigr\rangle \, dv_g \\
        &+ \int_{\Omega} \sum_{j=1}^m \Bigl\langle
     R^N \Big( f\overline{\nabla}_{e_j}\overline{\Delta}^{\,s-1}\tau(\phi), \overline{\Delta}^{\,s-1} \tau(\phi)\Big) d\phi(e_j)
        \,  ,  \,     V   \Bigr\rangle \, dv_g \\
  ={}& -\int_{\Omega} \bigl\langle \tau_{f,2s+1}(\phi),V\bigr\rangle \, dv_g.
\end{align*}
\end{small}

Recall that $f$-$k$-harmonic maps are critical points of $E_{f,k}(\phi)$, that is, they satisfy
\[
\left.\frac{d}{dt}\right|_{t=0} E_{f,k}(\phi_t) = 0
\]
for all compactly supported variations.
This demonstrates the desired result.
\end{proof}
\begin{remark}
When $k=1$, formulas \eqref{L-1-2} and \eqref{P-3} simplify respectively to
\[
   \overline{\nabla}_{\frac{\partial}{\partial t}} \overline{\nabla}_{e_i}\overline{\Delta}^{\,-1}\tau(\phi_t)\Big|_{t=0}
   = - \overline{\nabla}_{e_i} V ,
\]
and
\begin{align} \label{f-harmonic-1}
   \left.\frac{d}{dt}\right|_{t=0} E_{f,1}(\phi_t)
      &= \int_{\Omega} \sum_{i=1}^m \Bigl\langle
         \left.\overline{\nabla}_{\frac{\partial}{\partial t}} \overline{\nabla}_{e_i} \overline{\Delta}^{\,-1} \tau(\phi_t) \right|_{t=0},
         f\overline{\nabla}_{e_i}\overline{\Delta}^{\,-1}\tau(\phi)
         \Bigr\rangle \, dv_g  \\
      &= \int_{\Omega} \sum_{i=1}^m \Bigl\langle
         \overline{\nabla}_{e_i} V ,  f\,d\phi(e_i) \Bigr\rangle \, dv_g.  \nonumber
\end{align}
Applying \eqref{L-2-4} to \eqref{f-harmonic-1}, we obtain
\begin{equation*}
  \left.\frac{d}{dt}\right|_{t=0} E_{f,1}(\phi_t)
      = -\int_{\Omega} \bigl\langle V,\; f\tau(\phi) + d\phi(\nabla f) \bigr\rangle \, dv_g.
\end{equation*}
Consequently, the Euler-Lagrange equation yields
\begin{equation}\label{f-harmonic-2}
 \tau_{f,1}(\phi) := f\tau(\phi) + d\phi(\nabla f) = 0,
\end{equation}
which is precisely the $f$-harmonic map equation \eqref{f-tension}.
\end{remark}
\begin{remark}
Taking $k=2$, \eqref{f-2s} simplifies to
\begin{align}\label{f-harmonic-3}
  \tau_{f,2}(\phi) = & \ \overline{\Delta}\big( f\tau(\phi) \big)
    - R^N\big(f\tau(\phi), d\phi(e_j)\big)d\phi(e_j)  \\
   = & \ f\overline{\Delta}\tau(\phi) - (\Delta f)\tau(\phi)
        - 2 \overline{\nabla}_{\nabla f}\tau(\phi)
        - f R^N\big(\tau(\phi), d\phi(e_j)\big)d\phi(e_j)  \nonumber\\
   = & \ f\tau_2(\phi) - (\Delta f)\tau(\phi)
        - 2 \overline{\nabla}_{\nabla f}\tau(\phi), \nonumber
\end{align}
where $\tau_2(\phi) := \overline{\Delta}\tau(\phi)
- R^N\big(\tau(\phi), d\phi(e_j)\big)d\phi(e_j)$ is the usual bi-tension field.
This agrees with the $f$-biharmonic equation \eqref{f-biharmonic-equation}.
\end{remark}

\section{$f$-3-harmonic and $f$-4-harmonic curves}

In \cite{Ou2014-f-biharmonic}, Ou characterized $f$-biharmonic curves in general Riemannian manifolds. In this section, we study the existence of proper $f$-3-harmonic and $f$-4-harmonic curves of positive constant geodesic curvature in $N^2(C)$. For convenience, let $p\ge 2$ be an integer and set
\[
(\overline{\nabla}_{\gamma'})^p
:= \underbrace{\overline{\nabla}_{\gamma'} \overline{\nabla}_{\gamma'} \cdots \overline{\nabla}_{\gamma'}}_{p \text{ times}}.
\]

We begin by characterizing $f$-3-harmonic and $f$-4-harmonic curves in a general Riemannian manifold.
\begin{proposition}
Let $\gamma \colon I \to N^n$ be a smooth curve parametrized by arclength, and let $f \colon I \to \mathbb{R}^+$ be a smooth positive function. Then $\gamma$ is an $f$-3-harmonic curve if and only if
\begin{align}\label{3-harmonic-0}
 & f \Bigl\{ \bigl( \overline{\nabla}_{\gamma'} \bigr)^5 \gamma'
            + R^N\Bigl( \bigl( \overline{\nabla}_{\gamma'} \bigr)^3 \gamma',\, \gamma' \Bigr) \gamma'
            - R^N\Bigl( \bigl( \overline{\nabla}_{\gamma'} \bigr)^2 \gamma',\, \overline{\nabla}_{\gamma'} \gamma' \Bigr) \gamma' \Bigr\} \\
        &\quad + 3f' \bigl( \overline{\nabla}_{\gamma'} \bigr)^4 \gamma'
               + 3f'' \bigl( \overline{\nabla}_{\gamma'} \bigr)^3 \gamma'
            + f''' \bigl( \overline{\nabla}_{\gamma'} \bigr)^2 \gamma'  \nonumber\\
        &\quad + f'R^N\Bigl( \bigl( \overline{\nabla}_{\gamma'} \bigr)^2 \gamma',\, \gamma' \Bigr) \gamma'=0. \nonumber
\end{align}
\end{proposition}
\begin{proof}
Taking $s=1$ in \eqref{f-2s+1}, the $f$-3-tension field becomes
\begin{align} \label{3-harmonic-1}
  \tau_{f,3}(\phi)
  ={}&\overline{\Delta}\big( f \overline{\Delta}\tau(\phi) \big)
    -\overline{\Delta}\, \overline{\nabla}_{\nabla f} \,\tau(\phi) \\
     & -fR^N\Bigl(\overline{\nabla}_{e_j}\,\tau(\phi),
                  \tau(\phi)\Bigr) d\phi(e_j) \nonumber\\
     & -  R^N\Bigl(  f \overline{\Delta}\tau(\phi)-\overline{\nabla}_{\nabla f} \,\tau(\phi) ,
                         d\phi(e_j)\Bigr)d\phi(e_j) \nonumber \\
 ={}& f \tau_3 (\phi)
-(\Delta f)\overline{\Delta}\tau(\phi)
-2 \overline{\nabla}_{\nabla f} \overline{\Delta}\tau(\phi)
       -\overline{\Delta}\, \overline{\nabla}_{\nabla f} \,\tau(\phi) \nonumber\\
          & +  R^N\Bigl(  \overline{\nabla}_{\nabla f} \,\tau(\phi) ,
                         d\phi(e_j)\Bigr)d\phi(e_j), \nonumber
\end{align}
where the 3-tension field is defined as
\begin{align}\label{3-harmonic-2}
  \tau_3(\phi)=& \overline{\Delta}^{\, 2}\tau(\phi)
-R^N\big( \overline{\Delta}\tau(\phi),d\phi(e_j) \big)d\phi(e_j) \\
& \quad   - R^N\big( \overline{\nabla}_{e_j}\tau(\phi),\tau(\phi)\big)d\phi(e_j).  \nonumber
\end{align}
Then, by \eqref{3-harmonic-2} and the fact that $\tau(\gamma) = \overline{\nabla}_{\gamma'} \gamma'$ and $\overline{\Delta}=-\bigl( \overline{\nabla}_{\gamma'} \bigr)^2$, a direct calculation shows that
    \begin{align} \label{3-harmonic-3}
        \tau_3(\gamma)
        &= \bigl( \overline{\nabla}_{\gamma'} \bigr)^4 \tau(\phi)
            - R^N\Bigl( -\bigl( \overline{\nabla}_{\gamma'} \bigr)^2 \tau(\phi),\, \gamma' \Bigr) \gamma'
            - R^N\bigl( \overline{\nabla}_{\gamma'} \tau(\phi),\, \tau(\phi) \bigr) \gamma'  \\
        &= \bigl( \overline{\nabla}_{\gamma'} \bigr)^5 \gamma'
            + R^N\Bigl( \bigl( \overline{\nabla}_{\gamma'} \bigr)^3 \gamma',\, \gamma' \Bigr) \gamma'
            - R^N\Bigl( \bigl( \overline{\nabla}_{\gamma'} \bigr)^2 \gamma',\, \overline{\nabla}_{\gamma'} \gamma' \Bigr) \gamma'. \nonumber
    \end{align}
Applying \eqref{3-harmonic-3} to \eqref{3-harmonic-1} together with $d\gamma(\nabla f) = f'(s) \gamma'(s)$ leads to
    \begin{align*}
        \tau_{f,3}(\gamma)
        &= f \Bigl[ \bigl( \overline{\nabla}_{\gamma'} \bigr)^5 \gamma'
            + R^N\Bigl( \bigl( \overline{\nabla}_{\gamma'} \bigr)^3 \gamma',\, \gamma' \Bigr) \gamma'
            - R^N\Bigl( \bigl( \overline{\nabla}_{\gamma'} \bigr)^2 \gamma',\, \overline{\nabla}_{\gamma'} \gamma' \Bigr) \gamma' \Bigr]  \\
        &\quad + f'' \bigl( \overline{\nabla}_{\gamma'} \bigr)^3 \gamma'
            + 2f' \bigl( \overline{\nabla}_{\gamma'} \bigr)^4 \gamma'
            + \bigl( \overline{\nabla}_{\gamma'} \bigr)^2 \Bigl( f' \bigl( \overline{\nabla}_{\gamma'} \bigr)^2 \gamma' \Bigr) \\
        &\quad + R^N\Bigl( f' \bigl( \overline{\nabla}_{\gamma'} \bigr)^2 \gamma',\, \gamma' \Bigr) \gamma'. \\
        &= f \Bigl[ \bigl( \overline{\nabla}_{\gamma'} \bigr)^5 \gamma'
            + R^N\Bigl( \bigl( \overline{\nabla}_{\gamma'} \bigr)^3 \gamma',\, \gamma' \Bigr) \gamma'
            - R^N\Bigl( \bigl( \overline{\nabla}_{\gamma'} \bigr)^2 \gamma',\, \overline{\nabla}_{\gamma'} \gamma' \Bigr) \gamma' \Bigr] \\
        &\quad + 3f' \bigl( \overline{\nabla}_{\gamma'} \bigr)^4 \gamma'
              + 3f'' \bigl( \overline{\nabla}_{\gamma'} \bigr)^3 \gamma'
            + f''' \bigl( \overline{\nabla}_{\gamma'} \bigr)^2 \gamma'  \\
        &\quad + f'R^N\Bigl( \bigl( \overline{\nabla}_{\gamma'} \bigr)^2 \gamma',\, \gamma' \Bigr) \gamma'.
    \end{align*}
    This complete the proof.
\end{proof}
\begin{proposition}
Let $\gamma \colon I \to N^n$ be a smooth curve parametrized by arclength, and let $f \colon I \to \mathbb{R}^+$ be a smooth positive function. Then $\gamma$ is an $f$-4-harmonic curve if and only if
\begin{align} \label{4-harmonic-0}
 & f\Big\{- \bigl( \overline{\nabla}_{\gamma'} \bigr)^7 \gamma'
        - R^N\Bigl( \bigl( \overline{\nabla}_{\gamma'} \bigr)^5 \gamma', \gamma' \Bigr)\gamma' \\
 &\qquad + R^N\Bigl( \bigl( \overline{\nabla}_{\gamma'} \bigr)^4 \gamma', \overline{\nabla}_{\gamma'} \gamma' \Bigr)\gamma' \nonumber\\
 &\qquad - R^N\Bigl( \bigl( \overline{\nabla}_{\gamma'} \bigr)^3 \gamma', \bigl( \overline{\nabla}_{\gamma'} \bigr)^2 \gamma'\Bigr)\gamma' \Big\} \nonumber\\[4pt]
 & - f^{(4)} \bigl( \overline{\nabla}_{\gamma'} \bigr)^3 \gamma'
    - 4 f''' \bigl( \overline{\nabla}_{\gamma'} \bigr)^4 \gamma'
    - 6 f'' \bigl( \overline{\nabla}_{\gamma'} \bigr)^5 \gamma'
    - 4 f' \bigl( \overline{\nabla}_{\gamma'} \bigr)^6 \gamma' \nonumber\\[4pt]
 & - f'' \, R^N\Bigl( \bigl( \overline{\nabla}_{\gamma'} \bigr)^3 \gamma', \gamma' \Bigr)\gamma'
    - 2 f' \, R^N\Bigl( \bigl( \overline{\nabla}_{\gamma'} \bigr)^4 \gamma', \gamma' \Bigr)\gamma' \nonumber\\
 & + f' \, R^N\Bigl( \bigl( \overline{\nabla}_{\gamma'} \bigr)^3 \gamma', \overline{\nabla}_{\gamma'} \gamma' \Bigr)\gamma' = 0. \nonumber
\end{align}
\end{proposition}
\begin{proof}
For $k = 4$ (i.e., $s = 2$), the $f$-4-tension field obtained from \eqref{f-2s} reduces to
\begin{align}
  \tau_{f,4}(\phi)
  ={}& \overline{\Delta}^{\,2}\bigl(f \overline{\Delta}\tau(\phi)\bigr)
      - R^N\Bigl( \overline{\Delta}\bigl(f \overline{\Delta}\tau(\phi)\bigr),
                               d\phi(e_j)\Bigr)d\phi(e_j) \label{4-harmonic-1}\\
  &   -  R^N\Bigl( \overline{\nabla}_{e_j}\bigl( f \overline{\Delta}\tau(\phi)\bigr),
                       \tau(\phi)\Bigr)d\phi(e_j) \nonumber\\
  & + R^N\Bigl(  f \overline{\Delta}\tau(\phi),
                     \overline{\nabla}_{e_j}\,\tau(\phi)\Bigr)d\phi(e_j). \nonumber
\end{align}
A straightforward calculation shows
\begin{align} \label{4-harmonic-2}
  \tau_{f,4}(\phi)
  ={ } & f \overline{\Delta}^{\,3} \tau(\phi)
-2(\Delta f)\overline{\Delta}^{\,2} \tau(\phi)
-2\overline{\nabla}_{\nabla f} \overline{\Delta}^{\,2} \tau(\phi) \\
&   + (\Delta^2 f) \overline{\Delta}\tau(\phi)
  + 2 \overline{\nabla}_{ \nabla(\Delta f)}\overline{\Delta}\tau(\phi)
   - 2 \overline{\Delta}\, \overline{\nabla}_{\nabla f}\overline{\Delta}\tau(\phi) \nonumber\\
   &   - f  R^N\Bigl( \overline{\Delta}^{\,2} \tau(\phi) ,
                               d\phi(e_j)\Bigr)d\phi(e_j) \nonumber\\
   &   +(\Delta f)  R^N\Bigl(  \overline{\Delta}\tau(\phi) ,
                               d\phi(e_j)\Bigr)d\phi(e_j) \nonumber\\
   &   +2 R^N\Bigl(   \overline{\nabla}_{\nabla f}\overline{\Delta}\tau(\phi),
                               d\phi(e_j)\Bigr)d\phi(e_j) \nonumber\\
  &   -  R^N\Bigl( e_j(f)\overline{\Delta}\tau(\phi)
  +f\overline{\nabla}_{e_j}\overline{\Delta}\tau(\phi),
                       \tau(\phi)\Bigr)d\phi(e_j)  \nonumber\\
  & + f R^N\Bigl( \overline{\Delta}\tau(\phi),
                     \overline{\nabla}_{e_j}\,\tau(\phi)\Bigr)d\phi(e_j) \nonumber\\
    ={ } & f  \tau_4(\phi)
-2(\Delta f)\overline{\Delta}^{\,2} \tau(\phi)
-2\overline{\nabla}_{\nabla f} \overline{\Delta}^{\,2} \tau(\phi) \nonumber\\
&   + (\Delta^2 f) \overline{\Delta}\tau(\phi)
  + 2 \overline{\nabla}_{ \nabla(\Delta f)}\overline{\Delta}\tau(\phi)
   - 2 \overline{\Delta}\, \overline{\nabla}_{\nabla f}\overline{\Delta}\tau(\phi) \nonumber\\
   &   +(\Delta f)  R^N\Bigl(  \overline{\Delta}\tau(\phi) ,
                               d\phi(e_j)\Bigr)d\phi(e_j) \nonumber\\
   &   +2 R^N\Bigl(   \overline{\nabla}_{\nabla f}\overline{\Delta}\tau(\phi),
                               d\phi(e_j)\Bigr)d\phi(e_j) \nonumber\\
  &   -  R^N\Bigl( \overline{\Delta}\tau(\phi),
                       \tau(\phi)\Bigr)d\phi(\nabla f),  \nonumber
\end{align}
where the 4-tension field is defined as
\begin{align}  \label{4-harmonic-3}
  \tau_{4}(\phi)
  ={}& \overline{\Delta}^{\,3}\tau(\phi)
      - R^N\Bigl( \overline{\Delta}^{\,2}\tau(\phi),
                               d\phi(e_j)\Bigr)d\phi(e_j)\\
  &   -  R^N\Bigl( \overline{\nabla}_{e_j}\overline{\Delta}\tau(\phi),
                       \tau(\phi)\Bigr)d\phi(e_j) \nonumber\\
  &   + R^N\Bigl(  \overline{\Delta}\tau(\phi),
                     \overline{\nabla}_{e_j}\,\tau(\phi)\Bigr)d\phi(e_j).  \nonumber
\end{align}
Substituting $\tau(\gamma) = \overline{\nabla}_{\gamma'} \gamma'$ and $\overline{\Delta}=-\bigl( \overline{\nabla}_{\gamma'} \bigr)^2$ into \eqref{4-harmonic-3} again, we obtain its reduced form:
 \begin{align}  \label{4-harmonic-4}
  \tau_{4}(\gamma)
  ={}& - \bigl( \overline{\nabla}_{\gamma'} \bigr)^7 \gamma'
        - R^N\Bigl(  \bigl( \overline{\nabla}_{\gamma'} \bigr)^5 \gamma',
                               \gamma' \Bigr)\gamma'\\
  &   +  R^N\Bigl( \bigl( \overline{\nabla}_{\gamma'} \bigr)^4 \gamma',
                       \overline{\nabla}_{\gamma'}\gamma' \Bigr)\gamma' \nonumber\\
  &   - R^N\Bigl(  \bigl( \overline{\nabla}_{\gamma'} \bigr)^3 \gamma',
                    \bigl( \overline{\nabla}_{\gamma'} \bigr)^2 \gamma'\Bigr)\gamma'.  \nonumber
\end{align}
Inserting \eqref{4-harmonic-4} into \eqref{4-harmonic-2} and employing $d\gamma(\nabla f) = f'(s) \gamma'(s)$ leads to
\begin{align} \label{4-harmonic-5}
\tau_{f,4}(\gamma)
    ={ } & f\Big\{- \bigl( \overline{\nabla}_{\gamma'} \bigr)^7 \gamma'
        - R^N\Bigl( \bigl( \overline{\nabla}_{\gamma'} \bigr)^5 \gamma', \gamma' \Bigr)\gamma' \\
    &\quad + R^N\Bigl( \bigl( \overline{\nabla}_{\gamma'} \bigr)^4 \gamma', \overline{\nabla}_{\gamma'} \gamma' \Bigr)\gamma' \nonumber\\
    &\quad - R^N\Bigl( \bigl( \overline{\nabla}_{\gamma'} \bigr)^3 \gamma', \bigl( \overline{\nabla}_{\gamma'} \bigr)^2 \gamma'\Bigr)\gamma' \Big\} \nonumber\\[4pt]
    &\quad - 2f''\bigl( \overline{\nabla}_{\gamma'} \bigr)^5 \gamma'
      - 2f'\bigl( \overline{\nabla}_{\gamma'} \bigr)^6 \gamma'
      - f^{(4)}\bigl( \overline{\nabla}_{\gamma'} \bigr)^3 \gamma' \nonumber\\
    &\quad  - 2f'''\bigl( \overline{\nabla}_{\gamma'} \bigr)^4 \gamma'
     - 2 \bigl( \overline{\nabla}_{\gamma'} \bigr)^2 \bigl( f'\bigl( \overline{\nabla}_{\gamma'} \bigr)^4 \gamma'\bigr) \nonumber\\[4pt]
    &\quad - f'' R^N\Bigl( \bigl( \overline{\nabla}_{\gamma'} \bigr)^3 \gamma', \gamma' \Bigr)\gamma'
      - 2f' R^N\Bigl( \bigl( \overline{\nabla}_{\gamma'} \bigr)^4 \gamma', \gamma' \Bigr)\gamma' \nonumber\\
    &\quad + f' R^N\Bigl( \bigl( \overline{\nabla}_{\gamma'} \bigr)^3 \gamma', \overline{\nabla}_{\gamma'} \gamma' \Bigr)\gamma'. \nonumber
\end{align}
Noting that
$$\bigl( \overline{\nabla}_{\gamma'} \bigr)^2 \Bigl( f'\bigl( \overline{\nabla}_{\gamma'} \bigr)^4 \gamma'\Bigr)
= f''' (\overline{\nabla}_{\gamma'})^4 \gamma'
  + 2 f'' (\overline{\nabla}_{\gamma'})^5 \gamma'
  + f' (\overline{\nabla}_{\gamma'})^6 \gamma',$$
the above expression \eqref{4-harmonic-5} yields the desired result.
\end{proof}
For $n=2$, let $\gamma \colon I \to \big(N^2(C), \langle \cdot , \cdot \rangle\big)$ be a smooth curve parametrized by arclength in a 2-dimensional space form with orthonormal frame field $\{\overrightarrow{T},\overrightarrow{N}\}$ satisfying the Frenet formulas
\begin{equation}\label{Frenet-1}
  \overline{\nabla}_{\gamma'}\begin{pmatrix} \overrightarrow{T} \\ \overrightarrow{N} \end{pmatrix}
=\begin{pmatrix} 0  & \kappa \\ -\kappa & 0 \end{pmatrix}
\begin{pmatrix} \overrightarrow{T} \\ \overrightarrow{N} \end{pmatrix},
\end{equation}
where $\overrightarrow{T}:=\gamma'$ is the unit tangent vector field, $\overrightarrow{N}$ the unit normal vector field, and $\kappa$ the curvature of $\gamma$. If the curvature $\kappa$ is constant, we have
\begin{equation}\label{Frenet-2}
\left\{
\begin{array}{l}
\langle \overrightarrow{T},\overrightarrow{T} \rangle = \langle \overrightarrow{N},\overrightarrow{N} \rangle = 1,\quad \langle \overrightarrow{T},\overrightarrow{N} \rangle = 0; \\[6pt]
\bigl(\overline{\nabla}_{\gamma'}\bigr)^{2p} \gamma' = (-1)^{p}\kappa^{2p}\overrightarrow{T},\qquad p \geq 1; \\[6pt]
\bigl(\overline{\nabla}_{\gamma'}\bigr)^{2p+1} \gamma' = (-1)^{p}\kappa^{2p+1}\overrightarrow{N},\qquad p \geq 0.
\end{array}
\right.
\end{equation}

The following two theorems classify the $f$-3-harmonic and $f$-4-harmonic curves with positive constant geodesic curvature in a 2-dimensional space form $N^2(C)$.
\begin{theorem}\label{curve-3-theorem}
Let $\gamma \colon I \to N^2(C)$ be a proper $f$-3-harmonic curve parametrized by arclength from an open interval of $\mathbb{R}$ into a space form $N^2(C)$ with constant sectional curvature $C$, and let $f \colon I \to \mathbb{R}^+$ be a smooth positive function. Then, provided the curvature $\kappa>0$ is constant, it follows that $C<0$ and $\kappa=\frac{\sqrt{-C}}{2}$.
\end{theorem}
\begin{proof}
Substituting \eqref{Frenet-2} into the $f$-3-harmonic curve equation \eqref{3-harmonic-0} yields the following expansion:
\begin{align}\label{3-curve-1}
& f (-1)^{2}\kappa^5 \overrightarrow{N}
          +Cf \Bigl\{  (-1)^{1}\kappa^3 \overrightarrow{N}
          - \bigl\langle (-1)^{1}\kappa^3 \overrightarrow{N},\,
            \gamma' \bigr\rangle \gamma' \\
& \qquad - \bigl\langle \kappa\overrightarrow{N},\,
            \gamma' \bigr\rangle (-1)^{1}\kappa^2 \overrightarrow{T}
          + \bigl\langle (-1)^{1}\kappa^2 \overrightarrow{T},\,
            \gamma' \bigr\rangle \kappa\overrightarrow{N}
      \Bigr\} \nonumber\\
& \quad + 3f' (-1)^{2}\kappa^4 \overrightarrow{T}
      + 3f'' (-1)^{1}\kappa^3 \overrightarrow{N}
      + f''' (-1)^{1}\kappa^2 \overrightarrow{T} \nonumber\\
& \quad + Cf' \Bigl\{ (-1)^{1}\kappa^2 \overrightarrow{T}
                 - \bigl\langle (-1)^{1}\kappa^2 \overrightarrow{T},\,
                   \gamma' \bigr\rangle \gamma'
               \Bigr\} = 0. \nonumber
\end{align}
Since $\langle \overrightarrow{T}, \gamma' \rangle = 1$ and $\langle \overrightarrow{N}, \gamma' \rangle = 0$, the inner product terms simplify significantly. Specifically,
\[
\langle \kappa^3 \vec N, \gamma' \rangle = 0,\quad
\langle \kappa\vec N, \gamma' \rangle = 0,\quad
\langle \kappa^2 \vec T, \gamma' \rangle = \kappa^2.
\]
Applying these simplifications, equation \eqref{3-curve-1} reduces to
\begin{equation}\label{3-curve-2}
\kappa^3 \bigl( \kappa^2 f  - 2C f  - 3f''  \bigr) \overrightarrow{N}
+\kappa^2\bigl( 3\kappa^2 f'  - f'''  \bigr) \overrightarrow{T}= 0,
\end{equation}
which gives the system
\begin{equation}\label{3-curve-3}
\begin{cases}
(\kappa^2 - 2C)f  - 3f'' = 0,\\[4pt]
3\kappa^2 f' - f''' = 0.
\end{cases}
\end{equation}
Differentiating the first equation and substituting the second yields
\[
(\kappa^2 - 2C)f' = 3f''' = 9\kappa^2 f',
\]
which, for a non-trivial solution with $f' \not\equiv 0$, implies $\kappa^2 - 2C = 9\kappa^2$.
Hence, we obtain $\kappa^2 = -\dfrac{C}{4}$, i.e., $\kappa = \sqrt{-C}/2$ with $C<0$.
Substituting this relation back into the first equation of \eqref{3-curve-3} gives $f'' + \dfrac{3C}{4} f = 0$, which admits the positive solution
\begin{equation*}
f(s) = C_1 e^{\frac{\sqrt{-3C}}{2} s} + C_2 e^{-\frac{\sqrt{-3C}}{2} s},
\end{equation*}
where $C_1, C_2 > 0$ are arbitrary positive constants.

It should be noted that a 3-harmonic curve forces $f$ to be a positive constant. Combining this with \eqref{3-curve-3} yields $\kappa^2=2C$, contradicting the condition $\kappa^2=-C/4$. Therefore, for a curve $\gamma\colon I\to N^2(C)$ with positive constant geodesic curvature, $\gamma$ is a proper $f$-3-harmonic curve if and only if
$C<0$ and $\kappa=\frac{\sqrt{-C}}{2}.$
\end{proof}
\begin{theorem}\label{curve-4-theorem}
Let $\gamma \colon I \to N^2(C)$ be a proper $f$-4-harmonic curve parametrized by arclength from an open interval of $\mathbb{R}$ into a space form $N^2(C)$ with constant sectional curvature $C$, and let $f \colon I \to \mathbb{R}^+$ be a smooth positive function. Then, provided the curvature $\kappa>0$ is constant, it follows that $C<0$ and $\kappa=\frac{\sqrt{-2C}}{2}.$
\end{theorem}
\begin{proof}
After substituting \eqref{Frenet-2} into the $f$-4-harmonic curve equation \eqref{4-harmonic-0}, one arrives at
\begin{align}\label{4-curve-1}
& -f(-1)^3 \kappa^7 \overrightarrow{N}
+ Cf \Bigl\{ -(-1)^2 \kappa^5 \overrightarrow{N}
           + \big\langle (-1)^2 \kappa^5 \overrightarrow{N},\; \gamma' \big\rangle \gamma'  \\
& \qquad\quad + \big\langle \kappa \overrightarrow{N},\; \gamma' \big\rangle (-1)^2 \kappa^4 \overrightarrow{T}
           - \big\langle (-1)^2 \kappa^4 \overrightarrow{T},\; \gamma' \big\rangle \kappa \overrightarrow{N} \nonumber \\
& \qquad\quad - \big\langle (-1)^1 \kappa^2 \overrightarrow{T},\; \gamma' \big\rangle (-1)^1 \kappa^3 \overrightarrow{N}
           + \big\langle (-1)^1 \kappa^3 \overrightarrow{N},\; \gamma' \big\rangle (-1)^1 \kappa^2 \overrightarrow{T}
         \Bigr\}  \nonumber \\[6pt]
& - f^{(4)} (-1)^1 \kappa^3 \overrightarrow{N}
   - 4 f''' (-1)^2 \kappa^4 \overrightarrow{T}
   - 6 f'' (-1)^2 \kappa^5 \overrightarrow{N}
   - 4 f' (-1)^3 \kappa^6 \overrightarrow{T} \nonumber \\[6pt]
& - C f'' \Big( (-1)^1 \kappa^3 \overrightarrow{N}
            - \big\langle (-1)^1 \kappa^3 \overrightarrow{N},\; \gamma' \big\rangle \gamma' \Big) \nonumber \\
& - 2 C f' \Big( (-1)^2 \kappa^4 \overrightarrow{T}
            - \big\langle (-1)^2 \kappa^4 \overrightarrow{T},\; \gamma' \big\rangle \gamma' \Big) \nonumber \\
& + C f' \Big( \big\langle \kappa \overrightarrow{N},\; \gamma' \big\rangle (-1)^1 \kappa^3 \overrightarrow{N}
            - \big\langle (-1)^1 \kappa^3 \overrightarrow{N},\; \gamma' \big\rangle \kappa \overrightarrow{N} \Big) = 0. \nonumber
\end{align}
After simplification, we obtain
\begin{equation}\label{4-curve-2}
\kappa^3 \Bigl( f^{(4)} - 6 \kappa^2 f'' + C f'' + \kappa^4 f - 3C \kappa^2 f \Bigr) \overrightarrow{N}
+ 4\kappa^4 \Bigl( \kappa^2 f' - f''' \Bigr) \overrightarrow{T}
= 0,
\end{equation}
which is equivalent to the system
\begin{equation}\label{4-curve-3}
\begin{cases}
f^{(4)} - 6 \kappa^2 f'' + C f'' + \kappa^4 f - 3C \kappa^2 f = 0,\\[4pt]
\kappa^2 f' - f''' = 0.
\end{cases}
\end{equation}
Differentiating the second equation in \eqref{4-curve-3} and substituting into the first yields the reduced condition
\begin{equation}\label{4-curve-4}
(5\kappa^2 - C) f'' - \kappa^2 (\kappa^2 - 3C) f = 0.
\end{equation}
Now, differentiating \eqref{4-curve-4} and substituting $f''' = \kappa^2 f'$ from \eqref{4-curve-3} yields $\kappa = \frac{\sqrt{-2C}}{2}$ with $C<0$. Under this condition, \eqref{4-curve-4} simplifies to
$f'' + \dfrac{C}{2} f = 0,$
which admits the positive solution
\begin{equation*}
f(s) = C_3 e^{\frac{\sqrt{-2C}}{2} s} + C_4 e^{-\frac{\sqrt{-2C}}{2} s},
\end{equation*}
with $C_3, C_4 > 0$ arbitrary positive constants.

If $\gamma$ is a $4$-harmonic curve, then $f$ must be a positive constant. Substituting this into \eqref{4-curve-3} yields $\kappa^2=3C$, contradicting $\kappa^2=-C/2$. Consequently, for a curve $\gamma\colon I\to N^2(C)$ with positive constant geodesic curvature,
$\gamma$ is proper $f$-4-harmonic if and only if $C<0$ and $\kappa=\frac{\sqrt{-2C}}{2}$.
\end{proof}

\begin{corollary}
A 2-dimensional space form $N^2(C)$ with $C\ge 0$ admits no proper $f$-3-harmonic or $f$-4-harmonic curves of positive constant geodesic curvature.
\end{corollary}

\begin{remark}
In the $f$-biharmonic case, Ou proved in \cite{Ou2014-f-biharmonic} that every arclength-parametrized $f$-biharmonic curve with constant geodesic curvature in a space form $N^n(C)$ is necessarily biharmonic. In contrast, Theorems \ref{curve-3-theorem} and \ref{curve-4-theorem} above demonstrate the existence of arclength-parametrized proper $f$-3-harmonic and $f$-4-harmonic curves with positive constant geodesic curvature in the 2-dimensional hyperbolic space $\mathbb{H}^2$.
\end{remark}

\section{Examples of proper $f$-polyharmonic functions and maps}

The aim of this section is to present several examples of proper $f$-$k$-harmonic functions and maps.
Denote by $\mathbb{R}$ and $\mathbb{R}^+$ the sets of real and positive real numbers, respectively, and by $\mathbb{R}^n$ the $n$-dimensional Euclidean space with standard coordinates $x = (x_1, x_2, \dots, x_n)$.
Denote by $\mathbb{R}^n \setminus \{0\}$ the punctured Euclidean space (i.e., $\mathbb{R}^n$ with the origin removed), and set $r = |x| = \sqrt{\sum_{i=1}^n x_i^2}$, the distance from $x$ to the origin.

A smooth function $u: M \to \mathbb{R}$ is called $f$-harmonic if it satisfies the $f$-Laplace equation (see \cite{Ou2014-f-harmonic})
\begin{equation}\label{1-function}
 \Delta_f u:= \Delta u + \langle \nabla \ln f, \nabla u \rangle = 0.
\end{equation}

The following characterization of $f$-biharmonic functions was given by Ou (\cite{Ou2014-f-biharmonic}, Proposition 2.5).

\begin{proposition}\label{pro-1}{\rm(\cite{Ou2014-f-biharmonic})}
A smooth function $u: M \to \mathbb{R}$ on a Riemannian manifold $(M, \langle \cdot\,, \cdot \rangle)$ is $f$-biharmonic if and only if
\begin{equation}\label{2-function-1}
 \Delta_{f,2}u:= \Delta(f \Delta u) = 0,
\end{equation}
or equivalently,
\begin{equation}\label{2-function-2}
  f \Delta^2 u + (\Delta f) \Delta u + 2 \langle \nabla f, \nabla \Delta u \rangle = 0.
\end{equation}
\end{proposition}

The following result provides a characterization of $f$-3-harmonic functions on a Riemannian manifold.
\begin{proposition}\label{3-proposition}
A smooth function $u: M \to \mathbb{R}$ on a Riemannian manifold $(M, \langle \cdot\,,\cdot \rangle)$ is $f$-$3$-harmonic if and only if
\begin{equation}\label{3-function-1}
 \Delta_{f,3}u:= \Delta\big( \operatorname{div}(f \nabla \Delta u) \big) = 0,
\end{equation}
or equivalently,
\begin{equation}\label{3-function-2}
 f \Delta^3 u + (\Delta f)\Delta^2 u + 2\langle \nabla f, \nabla\Delta^2 u \rangle + \Delta\langle \nabla f, \nabla\Delta u \rangle = 0.
\end{equation}
\end{proposition}
\begin{proof}
Let $t$ be the standard coordinate on $\mathbb{R}$.
First observe that
\[
\tau(u)=(\Delta u)\frac{\partial}{\partial t},\qquad
\overline{\Delta} \tau(u)=-(\Delta^2 u)\frac{\partial}{\partial t}.
\]
Moreover, by the definition \eqref{3-harmonic-2} of the $3$-tension field we have
\[
\tau_3(u) = \overline{\Delta}^2 \tau(u) = (\Delta^3 u)\frac{\partial}{\partial t}.
\]
Substituting these into \eqref{3-harmonic-1} yields that $u$ is $f$-$3$-harmonic if and only if
\begin{equation}\label{3-function-3}
  f \Delta^3 u + (\Delta f)\Delta^2 u + 2\langle \nabla f, \nabla\Delta^2 u \rangle + \Delta\langle \nabla f, \nabla\Delta u \rangle = 0.
\end{equation}

Applying the Bochner formula to $\Delta(f \Delta^2 u)$ gives
\[
\Delta(f \Delta^2 u) = f \Delta^3 u + (\Delta f)\Delta^2 u + 2\langle \nabla f, \nabla\Delta^2 u \rangle,
\]
which leads to
\begin{equation}\label{3-function-4}
  \Delta\big(f \Delta^2 u + \langle \nabla f, \nabla\Delta u \rangle\big) = 0.
\end{equation}

A direct computation of the divergence term yields
\[
\operatorname{div}(f \nabla \Delta u) = f \Delta^2 u + \langle \nabla f, \nabla\Delta u \rangle.
\]
Combining this with \eqref{3-function-4} we obtain
\[
\Delta\big( \operatorname{div}(f \nabla \Delta u) \big) = 0.
\]
Hence, the proof is completed.
\end{proof}

\begin{remark}
Setting $N = \mathbb{R}$ in \eqref{f-2s} and \eqref{f-2s+1} yields the equations that define the so-called \emph{$f$-$k$-harmonic function}. More precisely, for an integer $k \ge 1$, let $s = \lfloor k/2 \rfloor$. A function $u$ is called $f$-$k$-harmonic if it satisfies
\begin{equation}
\Delta_{f,k}u:=
\begin{cases}
\Delta^s( f \,\Delta^s u) = 0, & k = 2s,\\[4pt]
\Delta^s\bigl( \operatorname{div}(f \nabla \Delta^s u) \bigr) = 0, & k = 2s+1,
\end{cases}
\label{eq:f-k-harmonic}
\end{equation}
where $s = 1,2,3,\dots$ in the even case and $s = 0,1,2,\dots$ in the odd case.
\end{remark}

Ou (\cite{Ou2014-f-harmonic, Ou2014-f-biharmonic, Ou2017}) has provided numerous examples of $f$-harmonic and $f$-2-harmonic functions and maps. A natural question arising from $f$-harmonic to $f$-2-harmonic functions is whether proper $f$-$k$-harmonic functions exist for $k \geq 3$. To facilitate the construction of further examples of proper $f$-polyharmonic functions and maps, we first state the following two key lemmas.

\begin{lemma}\label{compute-r}
Let $r = |x|$ denote the radial distance function on the punctured Euclidean space $\mathbb{R}^n \setminus \{0\}$, and let $h(r)$ be a smooth radial function depending only on $r$. Then the following identities obviously hold:
\[
\nabla h(r) = \frac{h'(r)}{r} x, \quad
\operatorname{div}\bigl( h(r) x \bigr) = n h(r) + r h'(r), \quad
\Delta h(r) = h''(r) + \frac{n-1}{r} h'(r).
\]
In particular, for any constant $\alpha \in \mathbb{R}$ and any positive integer $k$, we obtain
\begin{align}
\nabla r^\alpha &= \alpha r^{\alpha-2} x, \label{eq:grad}\\
\operatorname{div}(r^\alpha x) &= (n+\alpha) r^\alpha, \label{eq:div}\\
\Delta r^\alpha &= \alpha(n+\alpha-2) r^{\alpha-2}, \label{eq:laplacian}\\
\Delta^k r^\alpha &= \left( \prod_{j=0}^{k-1} (\alpha-2j)(\alpha+n-2-2j) \right) r^{\alpha-2k}, \label{eq:polygrad}\\
\Delta(x_i r^\alpha) &= \alpha(n+\alpha) x_i r^{\alpha-2}, \label{eq:laplacianxi}\\
\Delta^k (x_i r^\alpha) &= \left( \prod_{j=0}^{k-1} (\alpha-2j)(\alpha+n-2j) \right) x_i r^{\alpha-2k}. \label{eq:polyxi}
\end{align}
\end{lemma}

\begin{proof}
We verify formulas \eqref{eq:grad}-\eqref{eq:polyxi} by direct computation and induction.

\underline{Step 1. Proof of \eqref{eq:grad} and \eqref{eq:div}.}
Since $\partial_i r = x_i / r$, we have
\begin{equation}\label{eq:partial}
  \partial_i r^\alpha = \alpha r^{\alpha-1} \partial_i r = \alpha r^{\alpha-2} x_i,
\end{equation}
which immediately yields $\nabla r^\alpha = \alpha r^{\alpha-2} x$. For the divergence, using the product rule,
\[
\partial_i (r^\alpha x_i) = (\partial_i r^\alpha) x_i + r^\alpha \partial_i x_i = \alpha r^{\alpha-2} x_i^2 + r^\alpha.
\]
Summing over $i = 1,\dots,n$ gives
\[
\operatorname{div}(r^\alpha x) = \sum_{i=1}^n (\alpha r^{\alpha-2} x_i^2 + r^\alpha) = \alpha r^{\alpha-2} r^2 + n r^\alpha = (n+\alpha) r^\alpha.
\]

\underline{Step 2. Proof of \eqref{eq:laplacian} and \eqref{eq:polygrad}.}
Taking the divergence of $\nabla r^\alpha$ and applying \eqref{eq:div} with $\alpha$ replaced by $\alpha-2$, we obtain
\[
\Delta r^\alpha = \operatorname{div}(\nabla r^\alpha) = \operatorname{div}(\alpha r^{\alpha-2} x) = \alpha \bigl[ (n+\alpha-2) r^{\alpha-2} \bigr],
\]
which proves \eqref{eq:laplacian}. For the higher-order case \eqref{eq:polygrad}, we proceed by induction on $k$. The base case $k=1$ is exactly \eqref{eq:laplacian}. Assuming \eqref{eq:polygrad} holds for some $k \ge 1$, we apply $\Delta$ to both sides:
\[
\Delta^{k+1} r^\alpha = \Delta \left( \prod_{j=0}^{k-1} (\alpha-2j)(\alpha+n-2-2j) \cdot r^{\alpha-2k} \right).
\]
Since the product is a constant, using \eqref{eq:laplacian} with $\alpha$ replaced by $\alpha-2k$ yields
\[
\Delta^{k+1} r^\alpha = \prod_{j=0}^{k-1} (\alpha-2j)(\alpha+n-2-2j) \cdot (\alpha-2k)(\alpha-2k+n-2) r^{\alpha-2(k+1)}.
\]
This completes the induction.

\underline{Step 3. Proof of \eqref{eq:laplacianxi} and \eqref{eq:polyxi}.}
Setting \(u = x_i\) and \(v = r^{\alpha}\) in the product rule for the Laplacian:
\[
\Delta(uv) = u\Delta v + 2\nabla u \cdot \nabla v + v\Delta u,
\]
and noting that \(\Delta x_i = 0\) and \(\nabla x_i = e_i\) (the \(i\)-th standard basis vector), we obtain
\[
\Delta(x_i r^{\alpha}) = x_i \Delta r^{\alpha} + 2\,\partial_i r^{\alpha}.
\]
Substituting \eqref{eq:laplacian} and \eqref{eq:partial} gives
\[
\Delta(x_i r^\alpha) = x_i \cdot \alpha(n+\alpha-2) r^{\alpha-2} + 2 \cdot \alpha r^{\alpha-2} x_i = \alpha(n+\alpha) x_i r^{\alpha-2},
\]
which proves \eqref{eq:laplacianxi}. For the higher-order case \eqref{eq:polyxi}, we again proceed by induction on $k$. The base case $k=1$ is \eqref{eq:laplacianxi}. Assuming the formula holds for some $k \ge 1$, we apply $\Delta$ to both sides:
\[
\Delta^{k+1}(x_i r^\alpha) = \Delta\left( \prod_{j=0}^{k-1} (\alpha-2j)(\alpha+n-2j) \cdot x_i r^{\alpha-2k} \right).
\]
Using $\Delta(x_i r^{\beta}) = \beta(n+\beta) x_i r^{\beta-2}$ with $\beta = \alpha-2k$ (which follows from \eqref{eq:laplacianxi} by replacing $\alpha$ with $\beta$), we obtain
\[
\Delta^{k+1}(x_i r^\alpha) = \prod_{j=0}^{k-1} (\alpha-2j)(\alpha+n-2j) \cdot (\alpha-2k)(\alpha-2k+n) x_i r^{\alpha-2(k+1)}.
\]
This completes the induction.
\end{proof}

\begin{lemma}[Radial polyharmonic functions]
\label{lem:radial_polyharmonic}
Let \(n\ge 2\) and \(k\ge 1\) be integers, and let \(u\in C^{\infty}(\mathbb{R}^n\setminus\{0\})\) be a radial function, i.e., \(u(x)=f(r)\) where \(r=|x|>0\). Then \(\Delta^{k} u = 0\) in \(\mathbb{R}^n\setminus\{0\}\) if and only if \(f(r)\) takes one of the following forms:

\begin{itemize}
\item[(i)] \textbf{Case \(n=2\).}
\begin{equation}\label{eq:n2}
f(r)=\sum_{j=0}^{k-1} (C_j + D_j \ln r) r^{2j},\qquad C_j,D_j\in\mathbb{R}.
\end{equation}

\item[(ii)] \textbf{Case \(n\ge 3\).}
Define the exponent sets
\begin{equation*}
\Lambda_1=\{0,2,4,\dots,2k-2\},\qquad
\Lambda_2=\{2-n,4-n,\dots,2k-n\}.
\end{equation*}
\begin{itemize}
\item[(a)] If \(n\) is odd or \(n>2k\), then $\Lambda_1\cap \Lambda_2=\emptyset$, and
\begin{equation}\label{eq:n3-1}
f(r)=\sum_{j=0}^{k-1} A_j r^{2j}\;+\;\sum_{j=0}^{k-1} B_j r^{2j+2-n},\qquad A_j,B_j\in\mathbb{R}.
\end{equation}

\item[(b)] If \(n\) is even and \(4\le n\le 2k\), then $\Lambda_1\cap \Lambda_2=\{0,2,4,\cdots,2k-n\}$, and
\begin{equation}\label{eq:n3-2}
f(r)=\sum_{\lambda\in\Lambda_1\setminus\Lambda_2} A_\lambda r^{\lambda}
\;+\;\sum_{\mu\in\Lambda_2\setminus\Lambda_1} B_\mu r^{\mu}
\;+\;\sum_{\sigma\in\Lambda_1\cap\Lambda_2} r^{\sigma}\bigl(C_\sigma + D_\sigma \ln r\bigr),
\end{equation}
where \(A_\lambda,B_\mu,C_\sigma,D_\sigma\in\mathbb{R}\).
\end{itemize}
\end{itemize}

In the special case \(k=1\) (harmonic functions), every radial solution of \(\Delta u=0\) reduces to
\begin{equation}\label{eq:k=1}
u(x):=
f(r)=
\begin{cases}
C_0 + D_0 \ln r, & n=2,\\[4pt]
A_0 + B_0\, r^{2-n}, & n\ge 3.
\end{cases}
\end{equation}
\end{lemma}

\begin{proof}
Set \(t = \ln r\) (\(r>0\)) and define \(\varphi(t) = f(e^t)\).
A direct computation gives
\[
f'(r)=e^{-t}\varphi'(t),\qquad f''(r)=e^{-2t}\bigl(\varphi''(t)-\varphi'(t)\bigr).
\]
The radial Laplacian is
\[
\Delta f(r) = f''(r) + \frac{n-1}{r}f'(r)
= e^{-2t}\bigl(\varphi''(t)+(n-2)\varphi'(t)\bigr).
\]
Let $D=d/dt$, and denote $D^k=d^k/dt^k$ for $k\ge 2$. Set
$$P_1(D):=D^2+(n-2)D.$$ Then
\begin{equation}\label{eq:laplacian_transform}
 \Delta \varphi(t) = \Delta f(r) = e^{-2t} P_1(D) \varphi(t).
\end{equation}

\textbf{Claim.}
For any \(a\ge 0\) and any smooth function \(h(t)\),
\[
\Delta\bigl(e^{-2at}h(t)\bigr) = e^{-2(a+1)t}\Bigl[D^2 + (n-2-4a)D + 2a(2a+2-n)\Bigr]h(t).
\]
\noindent
\textit{Proof of the claim.}
Apply \eqref{eq:laplacian_transform} with \(\varphi(t)=e^{-2at}h(t)\). Compute
\[
\varphi' = e^{-2at}(h'-2ah),\quad
\varphi'' = e^{-2at}(h''-4ah'+4a^2h).
\]
Then
\[
\varphi''+(n-2)\varphi' = e^{-2at}\bigl[h'' + (n-2-4a)h' + 2a(2a+2-n)h\bigr],
\]
which together with \eqref{eq:laplacian_transform} gives the Claim.

Now assume for some \(k\ge 1\) that
\[
\Delta^k f = e^{-2kt} P_k(D)\varphi,
\]
with \(P_k(D)\) a constant-coefficient differential operator.
Applying the Claim with \(a=k\) and \(h = P_k(D)\varphi\) yields
\[
\Delta^{k+1}f = \Delta\bigl(e^{-2kt}P_k(D)\varphi\bigr)
= e^{-2(k+1)t}\Bigl[D^2 + (n-2-4k)D + 2k(2k+2-n)\Bigr]P_k(D)\varphi.
\]
Hence
\begin{align}\label{eq:recurrence}
  P_{k+1}(D) &= \bigl[D^2 + (n-2-4k)D + 2k(2k+2-n)\bigr] P_k(D) \\
             &= (D-2k)\bigl(D + n-2-2k\bigr) P_k(D). \nonumber
\end{align}

Starting from \(P_1(D)= (D-0)(D+n-2-0)\), iteration yields
\begin{equation}\label{eq:Pk}
  P_k(D) = \prod_{j=0}^{k-1} (D-2j)\bigl(D + n-2-2j\bigr).
\end{equation}
Consequently,
\begin{equation}\label{eq:Dkf}
  \Delta^k f = e^{-2kt} \prod_{j=0}^{k-1} (D-2j)\bigl(D + n-2-2j\bigr)\,\varphi(t).
\end{equation}
Since \(e^{-2kt}\neq 0\), the equation \(\Delta^k f = 0\) is equivalent to
\begin{equation}\label{eq:ode}
  \prod_{j=0}^{k-1} (D-2j)\bigl(D + n-2-2j\bigr)\,\varphi(t) = 0.
\end{equation}
This is a constant-coefficient linear ODE of order \(2k\). Its characteristic equation is
\begin{equation}\label{eq:char_eq}
  \prod_{j=0}^{k-1} (\nu-2j)\bigl(\nu + n-2-2j\bigr) = 0.
\end{equation}
The characteristic roots form two families:
\[
\Lambda_1 = \{2j \mid j=0,\dots,k-1\},\qquad
\Lambda_2 = \{2j+2-n \mid j=0,\dots,k-1\}.
\]

\medskip
\textbf{Case \(n=2\).}
Here \(n-2=0\), so the two families of roots coincide. Each root \(2j\) (\(j=0,\dots,k-1\)) appears with multiplicity \(2\). The characteristic equation \eqref{eq:char_eq} becomes
\[
\prod_{j=0}^{k-1} (\nu-2j)^2 = 0.
\]
The general solution of \eqref{eq:ode} is
\[
\varphi(t) = \sum_{j=0}^{k-1} (C_j + D_j t) e^{2jt},\qquad C_j, D_j \in \mathbb{R}.
\]
Substituting \(t = \ln r\) gives
\begin{equation}\label{eq:solution_n2}
  f(r) = \sum_{j=0}^{k-1} (C_j + D_j \ln r) r^{2j},
\end{equation}
which matches the statement for \(n=2\).

\medskip
\textbf{Case \(n\ge 3\).}

\textbf{Subcase (a):}
If \(\Lambda_1\cap\Lambda_2 = \varnothing\) (i.e., \(n\) is odd or \(n>2k\)), all roots are simple. The general solution of \eqref{eq:ode} is
\[
\varphi(t) = \sum_{j=0}^{k-1} A_j e^{2jt} + \sum_{j=0}^{k-1} B_j e^{(2j+2-n)t},\qquad A_j,B_j\in\mathbb{R}.
\]
Returning to \(r=e^t\) gives
\begin{equation}\label{eq:solution_nonresonant}
  f(r) = \sum_{j=0}^{k-1} A_j r^{2j} + \sum_{j=0}^{k-1} B_j r^{2j+2-n}.
\end{equation}

\textbf{Subcase (b):}
If \(\Lambda_1\cap\Lambda_2 \neq \varnothing\) (i.e., \(n\) is even and \(4\le n\le 2k\)), then each common root \(\nu\) has multiplicity \(2\). For such a root, the solution contains a term \(e^{\nu t}(C + D t)\). For roots that belong to only one of the sets, the multiplicity is \(1\). Hence
\[
\varphi(t) = \sum_{\lambda\in\Lambda_1\setminus\Lambda_2} A_\lambda e^{\lambda t}
+ \sum_{\mu\in\Lambda_2\setminus\Lambda_1} B_\mu e^{\mu t}
+ \sum_{\sigma\in\Lambda_1\cap\Lambda_2} e^{\sigma t}\bigl(C_\sigma + D_\sigma t\bigr).
\]
Substituting back \(t=\ln r\) yields
\begin{equation}\label{eq:solution_resonant}
  f(r) = \sum_{\lambda\in\Lambda_1\setminus\Lambda_2} A_\lambda r^{\lambda}
+ \sum_{\mu\in\Lambda_2\setminus\Lambda_1} B_\mu r^{\mu}
+ \sum_{\sigma\in\Lambda_1\cap\Lambda_2} r^{\sigma}\bigl(C_\sigma + D_\sigma \ln r\bigr).
\end{equation}
This is exactly the form stated in the lemma for \(n\ge 3\).
\end{proof}

\begin{remark}
For every $n\ge 2$, the space of radial $k$-harmonic functions has dimension $2k$.
Logarithmic terms arise only in the \emph{resonant} case, i.e., when the two families of exponents
\[
\Lambda_1 = \{2j_1 \mid 0\le j_1\le k-1\},\qquad
\Lambda_2 = \{2j_2+2-n \mid 0\le j_2\le k-1\}
\]
satisfy $\Lambda_1\cap\Lambda_2\neq\varnothing$. This occurs precisely for even $n$ with $2\le n\le 2k$.

Lemma~\ref{lem:radial_polyharmonic} is essentially the radial version of the \emph{Almansi expansion} (cf.~\cite{Aronszajn1983}, Chapter~1.1), which states that any $k$-harmonic function in $\mathbb{R}^n\setminus\{0\}$ can be expressed as
\[
u(x) = \sum_{j=0}^{k-1} r^{2j} u_j(x),
\]
where each $u_j$ is harmonic. In the radial setting, this leads to the following:

\begin{itemize}
\item \textbf{Non-resonant case ($n$ odd or $n>2k$):}
The solution is a linear combination of pure powers as in \eqref{eq:n3-1}.

\item \textbf{Resonant case ($n$ even, $2\le n\le 2k$):}
Each coincident exponent yields at most a term linear in $\ln r$, and thus no higher-order terms appear.
\end{itemize}
\end{remark}

\begin{example}\label{ex-1}
For \(n \geq 3\), define a function \(u\) on \(\mathbb{R}^n \setminus \{0\}\) by \(u(x) = \ln r\).

\begin{enumerate}
    \item If we define a positive function \(f\) on \(\mathbb{R}^n \setminus \{0\}\) by
\[
f(x) = r^{2-n},
\]
then \(u\) is a proper \(f\)-harmonic function on \(\mathbb{R}^n \setminus \{0\}\).

    \item If we define a positive function \(f\) on \(\mathbb{R}^n \setminus \{0\}\) by
\[
f(x) = A r^2 + B r^{4-n}
\]
with \(A, B > 0\) and \(n \neq 4\), then \(u\) is a proper \(f\)-biharmonic function on \(\mathbb{R}^n \setminus \{0\}\).

    \item If we define a positive function \(f\) on \(\mathbb{R}^n \setminus \{0\}\) by
\[
f(x) = A r^{6-n} + B r^{4-n}+Cr^4
\]
with \(A, B, C > 0\) and \(n \notin \{4,6\}\), then \(u\) is a proper \(f\)-3-harmonic function on \(\mathbb{R}^n \setminus \{0\}\).
\end{enumerate}
\end{example}

\begin{proof}
A direct calculation shows
\[
\nabla u = \frac{x}{r^2},\quad
\Delta u = u''(r) + \frac{n-1}{r}u'(r) = \frac{n-2}{r^2},
\quad \nabla\Delta u = -2(n-2)\frac{x}{r^4},
\]
\[
\Delta^2 u = (n-2)\left[\left(\frac{1}{r^2}\right)'' + \frac{n-1}{r}\left(\frac{1}{r^2}\right)'\right]
= -\frac{2(n-2)(n-4)}{r^4},
\]
\[
\Delta^3 u = 2(n-2)(4-n)\left[\left(\frac{1}{r^4}\right)'' + \frac{n-1}{r}\left(\frac{1}{r^4}\right)'\right]
= \frac{8(n-2)(n-4)(n-6)}{r^6}.
\]
Hence:
\begin{itemize}
    \item $u$ is not harmonic for any $n\ge 3$;
    \item $u$ is not biharmonic for any $n\ge 3$ with $n\neq 4$;
    \item $u$ is not $3$-harmonic for any $n\ge 3$ with $n\notin\{4,6\}$.
\end{itemize}
\begin{enumerate}
    \item If \(f(x) = r^{2-n}\), then \(\nabla f = (2-n)\dfrac{x}{r^n}\). Consequently,
\[
\Delta u + \langle \nabla \ln f, \nabla u \rangle
= \frac{n-2}{r^2} + \frac{1}{r^{2-n}}\left\langle (2-n)\frac{x}{r^n}, \frac{x}{r^2} \right\rangle = 0,
\]
which, together with \eqref{1-function}, shows that \(u\) is proper \(f\)-harmonic.

    \item If \(f(x) = A r^2 + B r^{4-n}\) with \(A, B > 0\) and \(n \neq 4\), then
\[
\Delta\bigl(f\Delta u\bigr) = \Delta\!\left(A(n-2) + B(n-2)r^{2-n}\right) = 0.
\]
By Proposition \ref{pro-1}, \(u\) is therefore proper \(f\)-biharmonic.

    \item If \(f(x) = A r^{6-n} + B r^{4-n}+Cr^4\) with \(A, B, C > 0\) and \(n \notin \{4,6\}\), note that
\[
\nabla\Delta u = -2(n-2)\frac{x}{r^4}.
\]
A direct calculation then yields
\[
f\nabla\Delta u = -2(n-2)\bigl(A r^{2-n} + B r^{-n}+C\bigr)x.
\]
Taking the divergence gives
\[
\operatorname{div}(f\nabla\Delta u) =-2(n-2)(2A r^{2-n}+nC).
\]
Finally,
\[
\Delta\bigl(\operatorname{div}(f\nabla\Delta u)\bigr) = -2(n-2)\Delta(2A r^{2-n}+nC) = 0.
\]
Therefore, by Proposition \ref{3-proposition}, \(u\) is proper \(f\)-3-harmonic.
\end{enumerate}
\end{proof}

\begin{remark}\label{rem:ex-1-general}
Let $s$ be a positive integer and let $n \geq 3$ with $n \notin \{4,6,\dots,2s\}$.
For $u(x)=\ln r$ on $\mathbb{R}^n\setminus\{0\}$ (see Example~\ref{ex-1}), from Lemma~\ref{compute-r} we have
\[
\Delta^s u = (-2)^{s-1}\, (s-1)! \left(\prod_{j=1}^{s}(n-2j)\right) r^{-2s},
\]
\[
\nabla\Delta^s u = (-2)^{s}\, s! \left(\prod_{j=1}^{s}(n-2j)\right) r^{-2(s+1)} x.
\]

Define two positive functions $f_1, f_2$ on $\mathbb{R}^n\setminus\{0\}$ by
\[
f_1(x) = \Bigl( \sum_{i=0}^{s-1} A_i r^{2i} + \sum_{i=1}^{s} B_i r^{2i-n} \Bigr) r^{2s},
\]
\[
f_2(x) = \Bigl( \sum_{i=0}^{s-1} C_i r^{2i} + D_0 r^{-n} + \sum_{i=1}^{s} D_i r^{2i-n} \Bigr) r^{2(s+1)},
\]
where $A_i, B_i, C_i, D_i > 0$ are constants.

Consequently, applying Lemma~\ref{lem:radial_polyharmonic} together with \eqref{eq:f-k-harmonic},
one verifies that $u = \ln r$ is a proper $f_1$-$2s$-harmonic function
and also a proper $f_2$-$(2s+1)$-harmonic function on $\mathbb{R}^n\setminus\{0\}$.
\end{remark}

\begin{example}\label{ex-2}
For \(n \geq 2\), define \(u: \mathbb{R}^n \setminus \{0\} \to \mathbb{R}\) by \(u(x) = r^{-n}\).
\begin{enumerate}
    \item Let \(f: \mathbb{R}^n \setminus \{0\} \to \mathbb{R}^+\) be defined by
\[
f(x) = r^2.
\]
Then \(u\) is a proper \(f\)-harmonic function.

    \item Let \(f: \mathbb{R}^n \setminus \{0\} \to \mathbb{R}^+\) be defined by
\[
f(x) = A r^{n+2} + B r^4,
\]
where \(A, B > 0\) are constants. Then \(u\) is a proper \(f\)-biharmonic function.

    \item Let \(f: \mathbb{R}^n \setminus \{0\} \to \mathbb{R}^+\) be defined by
\[
f(x) = A r^6 + B r^4 + C r^{n+4},
\]
where \(A, B, C > 0\) are constants. Then \(u\) is a proper \(f\)-3-harmonic function.
\end{enumerate}
\end{example}

\begin{proof}
A direct computation yields
\[
\nabla u = -\frac{nx}{r^{n+2}},\quad
\Delta u = u''(r) + \frac{n-1}{r}u'(r) = \frac{2n}{r^{n+2}},\quad
\nabla\Delta u = -2n(n+2)\frac{x}{r^{n+4}},
\]
\[
\Delta^2 u = 2n\left[\left(\frac{1}{r^{n+2}}\right)'' + \frac{n-1}{r}\left(\frac{1}{r^{n+2}}\right)'\right]
= \frac{8n(n+2)}{r^{n+4}},
\]
\[
\Delta^3 u = 8n(n+2)\left[\left(\frac{1}{r^{n+4}}\right)'' + \frac{n-1}{r}\left(\frac{1}{r^{n+4}}\right)'\right]
= \frac{48n(n+2)(n+4)}{r^{n+6}}.
\]
Hence, \(u\) is neither harmonic, nor \(2\)-harmonic, nor \(3\)-harmonic.

\begin{enumerate}
    \item Let \(f(x) = r^2\). Then \(\nabla f = 2x\). Consequently,
\[
\Delta u + \langle \nabla \ln f, \nabla u \rangle
= \frac{2n}{r^{n+2}} + \frac{1}{r^2}\left\langle 2x, -\frac{nx}{r^{n+2}} \right\rangle = 0,
\]
which, together with \eqref{1-function}, shows that \(u\) is a proper \(f\)-harmonic function.

    \item Let \(f(x) = A r^{n+2} + B r^4\) with \(A, B > 0\). Then
\[
\Delta\bigl(f\Delta u\bigr) = 2n\Delta\!\left(A + B r^{2-n}\right) = 0.
\]
By Proposition \ref{pro-1}, \(u\) is therefore a proper \(f\)-biharmonic function.

    \item Let \(f(x) = A r^6 + B r^4 + C r^{n+4}\) with \(A, B, C > 0\). Since
\[
\nabla\Delta u = -2n(n+2)\frac{x}{r^{n+4}},
\]
we obtain
\[
f\nabla\Delta u = -2n(n+2)\bigl(A r^{2-n} + B r^{-n} + C\bigr)x.
\]
Taking the divergence gives
\[
\operatorname{div}(f\nabla\Delta u) = -2n(n+2)(2A r^{2-n} + nC).
\]
Finally,
\[
\Delta\bigl(\operatorname{div}(f\nabla\Delta u)\bigr) = -2n(n+2)\Delta(2A r^{2-n} + nC) = 0.
\]
Therefore, by Proposition \ref{3-proposition}, \(u\) is a proper \(f\)-3-harmonic function.
\end{enumerate}
\end{proof}

\begin{remark}\label{rem:ex-2-general}
Let $n \geq 2$ and $s \in \mathbb{N}^*$.
For $u(x)=r^{-n}$ on $\mathbb{R}^n\setminus\{0\}$ (see Example~\ref{ex-2}), Lemma~\ref{compute-r} yields
\[
\Delta^s u = 2^{s} s! \prod_{j=0}^{s-1}(n+2j) \; r^{-n-2s},
\qquad
\nabla\Delta^s u = -2^{s} s! \prod_{j=0}^{s}(n+2j) \; r^{-n-2(s+1)} x.
\]

Consider first the case where $n$ is odd or $n > 2s$.
Define positive functions $f_1, f_2$ on $\mathbb{R}^n\setminus\{0\}$ by
\begin{align*}
f_1(x) &= \Bigl( \sum_{i=0}^{s-1} A_i r^{2i} + \sum_{i=1}^{s} B_i r^{2i-n} \Bigr) r^{n+2s},\\
f_2(x) &= \Bigl( \sum_{i=0}^{s-1} C_i r^{2i} + D_0 r^{-n} + \sum_{i=1}^{s} D_i r^{2i-n} \Bigr) r^{n+2(s+1)},
\end{align*}
with $A_i, B_i, C_i, D_i > 0$.
Then, by Lemma~\ref{lem:radial_polyharmonic} and \eqref{eq:f-k-harmonic},
$u$ is a proper $f_1$-$2s$-harmonic and $f_2$-$(2s+1)$-harmonic function on $\mathbb{R}^n\setminus\{0\}$.

Now suppose that $n$ is even and $2 \leq n \leq 2s$.
Define positive functions $f_3, f_4$ on $\mathbb{R}^n\setminus\{0\}$ by
\begin{align*}
f_3(x) &= \Bigl( \sum_{i=0}^{s-1} A_i r^{2i} + \sum_{i=1}^{\frac{n}{2}-1} B_i r^{2i-n} \Bigr) r^{n+2s},\\
f_4(x) &= \Bigl( \sum_{i=0}^{s-1} C_i r^{2i} + D_0 r^{-n} + \sum_{i=1}^{\frac{n}{2}-1} D_i r^{2i-n} \Bigr) r^{n+2(s+1)},
\end{align*}
with $A_i, B_i, C_i, D_i > 0$.
Then, by Lemma~\ref{lem:radial_polyharmonic} and \eqref{eq:f-k-harmonic},
$u$ is a proper $f_3$-$2s$-harmonic and $f_4$-$(2s+1)$-harmonic function on $\mathbb{R}^n\setminus\{0\}$.
\end{remark}

\begin{example}[Proper $f$-$k$-harmonic maps of the form $\phi(x)=x r^{\alpha}$]
\label{ex:radial-power-f-k-harmonic}

Let $\phi: \mathbb{R}^n \setminus \{0\} \to \mathbb{R}^n \setminus \{0\}$ be defined by
\[
\phi(x) = x \, r^{\alpha},
\]
with $\alpha \in \mathbb{R}$, and let $f(x) = r^{\beta}$ for some $\beta \in \mathbb{R}$.
Then $\phi$ is a proper $f$-$k$-harmonic map if and only if the parameters $\alpha$ and $\beta$ satisfy the following conditions, depending on the parity of $k$.

\textbf{Case 1: $k = 2s$ with $s \ge 1$.}
The map $\phi$ is proper $f$-$k$-harmonic if and only if
\[
\alpha \notin \bigl\{ 2j,\; 2j - n \mid j = 0,1,\dots,2s-1 \bigr\}
\]
and
\[
\alpha + \beta \in \bigl\{ 2j + 2s,\; 2j + 2s - n \mid j = 0,1,\dots,s-1 \bigr\}.
\]

\textbf{Case 2: $k = 2s+1$ with $s \ge 1$.}
Let $D(q,\beta) := q(n+q+\beta) + \beta$ with $q = \alpha - 2s$.
Then $\phi$ is proper $f$-$k$-harmonic if and only if
\[
\alpha \notin \bigl\{ 2j,\; 2j - n \mid j = 0,1,\dots,2s \bigr\}
\]
and
\[
D(\alpha-2s,\beta) = 0 \quad\text{or}\quad
\alpha + \beta \in \bigl\{ 2j + 2s + 2,\; 2j + 2s + 2 - n \mid j = 0,1,\dots,s-1 \bigr\}.
\]

\end{example}

\begin{proof}
Fix a coordinate index $i\in\{1,\dots,n\}$ and consider the component function
\[
u(x) := \phi_i(x) = x_i r^{\alpha}.
\]
Let $\phi = (\phi_1,\dots,\phi_n)$ with each $\phi_i$ defined as above.
The map $\phi$ is called \emph{proper $f$-$k$-harmonic} if each component $u$ satisfies the $f$-$k$-harmonic equation \eqref{eq:f-k-harmonic} and, in addition, $u$ is not $k$-harmonic, i.e., $\Delta^k u \neq 0$.

Applying \eqref{eq:polyxi} in Lemma \ref{compute-r} to $u = x_i r^{\alpha}$, we obtain
\begin{equation}\label{As}
  \Delta^s u = A_s(\alpha) \, x_i r^{\alpha-2s},
\end{equation}
where
\begin{equation}\label{As-definition}
  A_s(\alpha) := \prod_{j=0}^{s-1} (\alpha-2j)(\alpha+n-2j),
\end{equation}
with the convention $A_0(\alpha) \equiv 1$.

\medskip
\textbf{Case 1.}
If $k = 2s\geq 2$, then the proper $f$-$k$-harmonic equation for $u$ consists of the two conditions
\begin{equation}\label{equation-2s}
  \Delta^s\bigl( f \,\Delta^s u \bigr) = 0 \qquad\text{and}\qquad \Delta^{2s}u = A_{2s}(\alpha) \, x_i r^{\alpha-4s} \neq 0.
\end{equation}

Multiplying \eqref{As} by $f = r^{\beta}$ gives
\[
f \Delta^s u = A_s(\alpha) \, x_i r^{\alpha+\beta-2s}.
\]
Set $p := \alpha + \beta - 2s$. Applying \eqref{As} once more yields
\[
\Delta^s(f \Delta^s u) = A_s(\alpha) A_s(p) \, x_i r^{p-2s}.
\]
Hence the two requirements in \eqref{equation-2s} become
\[
A_s(\alpha) A_s(p) = 0 \qquad\text{and}\qquad A_{2s}(\alpha) \neq 0,
\]
respectively.

The condition $A_{2s}(\alpha) \neq 0$ automatically implies $A_s(\alpha) \neq 0$ because
$A_{2s}(\alpha) = A_s(\alpha) A_s(\alpha-2s)$.
Consequently we must have $A_s(p) = 0$.

By the definition of $A_s$ in \eqref{As-definition}, $A_s(p)=0$ means there exists an integer $j_0$ with $0 \le j_0 \le s-1$ such that
\[
p = 2j_0 \quad\text{or}\quad p =  2j_0 -n.
\]
Substituting $p = \alpha+ \beta - 2s $ gives the equivalent condition
\[
\alpha + \beta = 2j_0 + 2s \quad\text{or}\quad \alpha + \beta = 2j_0 + 2s - n, \qquad 0 \le j_0 \le s-1.
\]

Therefore, for $k = 2s \ge 2$, the map $\phi$ is proper $f$-$k$-harmonic if and only if
\[
\alpha \notin \bigl\{ 2j,\; 2j - n \mid j = 0,1,\dots,2s-1 \bigr\}
\]
and
\[
\alpha + \beta \in \bigl\{ 2j + 2s,\; 2j + 2s - n \mid j = 0,1,\dots,s-1 \bigr\}.
\]

\medskip
\textbf{Case 2.}
If $k = 2s+1\geq 3$, then the proper $f$-$k$-harmonic equation for $u$ becomes
\begin{equation}\label{equation-2s+1}
  \Delta^s\bigl( \operatorname{div}(f \nabla \Delta^s u) \bigr) = 0 \qquad\text{and}\qquad \Delta^{2s+1}u = A_{2s+1}(\alpha) \, x_i r^{\alpha-4s-2} \neq 0.
\end{equation}

Let $v := \Delta^s u$. From \eqref{As} we have $v = A_s(\alpha) \, x_i r^{q}$ with $q := \alpha - 2s$.
A direct computation of the divergence gives
\[
\operatorname{div}\bigl( r^{\beta} \nabla (x_i r^{q}) \bigr) = D(q,\beta) \, x_i r^{\beta+q-2},
\]
where
\[
D(q,\beta) := q(n+q+\beta) + \beta.
\]
Indeed, using $\partial_j(x_i r^{q}) = \delta_{ij} r^{q} + q x_i x_j r^{q-2}$, we obtain
\[
\sum_{j=1}^n \partial_j\bigl( \delta_{ij} r^{\beta+q} \bigr) = (\beta+q) x_i r^{\beta+q-2},
\]
\[
\sum_{j=1}^n \partial_j\bigl( q x_i x_j r^{\beta+q-2} \bigr) = q \bigl( n + \beta+q-1 \bigr) x_i r^{\beta+q-2},
\]
and adding these two expressions yields
\[
\bigl[ (\beta+q) + q(n+\beta+q-1) \bigr] x_i r^{\beta+q-2} = D(q,\beta) \, x_i r^{\beta+q-2}.
\]

Consequently,
\[
\operatorname{div}(f \nabla v) = A_s(\alpha) D(q,\beta) \, x_i r^{\beta+q-2}.
\]
Set $p' := \beta + q - 2 = \beta + \alpha - 2s - 2$. Applying \eqref{As} we obtain
\[
\Delta^s\bigl( \operatorname{div}(f \nabla v) \bigr) = A_s(\alpha) D(q,\beta) A_s(p') \, x_i r^{p'-2s}.
\]

Thus, the two requirements in \eqref{equation-2s+1} reduce to
\[
A_s(\alpha) D(q,\beta) A_s(p') = 0 \qquad\text{and}\qquad A_{2s+1}(\alpha) \neq 0,
\]
respectively.

Since $A_{2s+1}(\alpha)\neq 0$ implies $A_s(\alpha)\neq 0$, we have $D(q,\beta)A_s(p')=0$. By the definition of $A_s(\alpha)$ in \eqref{As-definition}, $A_s(p') = 0$ means there exists an integer $j_0$ with $0 \le j_0 \le s-1$ such that
\[
p' = 2j_0 \quad\text{or}\quad p' = 2j_0 - n.
\]
Substituting $p' = \beta + \alpha - 2s - 2$ gives the equivalent condition
\[
\alpha + \beta = 2j_0 + 2s + 2 \quad\text{or}\quad \alpha + \beta = 2j_0 + 2s + 2 - n, \qquad 0 \le j_0 \le s-1.
\]

Therefore, for $k = 2s+1 \ge 3$, the map $\phi$ is proper $f$-$k$-harmonic if and only if
\[
\alpha \notin \bigl\{ 2j,\; 2j - n \mid j = 0,1,\dots,2s \bigr\}
\]
and
\[
D(\alpha-2s,\beta) = 0 \quad\text{or}\quad \alpha + \beta \in \bigl\{ 2j + 2s + 2,\; 2j + 2s + 2 - n \mid j = 0,1,\dots,s-1 \bigr\}.
\]
The proof is now complete.
\end{proof}

\begin{remark}
\textbf{The case $k=1$ ($s=0$).}
In this case $A_0(\alpha) \equiv 1$, $q = \alpha$, and equation \eqref{equation-2s+1} reduces to
\[
\operatorname{div}(f \nabla u) = 0 \quad \text{and} \quad \Delta u = \alpha(n+\alpha)x_i r^{\alpha-2} \neq 0.
\]
Using $\operatorname{div}(f\nabla u)=0$, a direct computation gives
\[
D(\alpha, \beta):=\alpha(n+\alpha+\beta) + \beta = 0.
\]
Properness excludes the harmonic case $\alpha(n+\alpha)=0$, i.e., $\alpha \notin \{0, -n\}$.
Therefore, the map $\phi: \mathbb{R}^n \setminus \{0\} \to \mathbb{R}^n \setminus \{0\}$ defined by $\phi(x)=x\,r^{\alpha}$ is a proper $f$-harmonic map with $f(x)=r^{\beta}$ if and only if
\[
\alpha(n+\alpha+\beta) + \beta = 0 \quad \text{and} \quad \alpha \notin \{0,-n\}.
\]
\end{remark}

\begin{remark}
\textbf{The case $k=2$ ($s=1$).}
By Case~1 in Example~\ref{ex:radial-power-f-k-harmonic}, the map $\phi: \mathbb{R}^n \setminus \{0\} \to \mathbb{R}^n \setminus \{0\}$ defined by $\phi(x)=x r^{\alpha}$ is a proper $f$-biharmonic map with $f(x)=r^{\beta}$ if and only if
\[
\alpha \notin \{0,2,-n,2-n\}
\]
and
\[
\alpha+\beta \in \{2,\,2-n\}.
\]
In $\mathbb{R}^n$, the inversion with respect to the unit sphere is given by
\[
\phi(x)=x r^{-2}, \qquad x\neq 0,
\]
so that $\alpha=-2$. If $n=2$, then $\phi$ is harmonic; if $n=4$, then $\phi$ is proper biharmonic; if $n\neq 2,4$, then $\phi$ is a proper $f$-biharmonic map with $f=r^4$ or $f=r^{4-n}$.
This characterization of proper $f$-biharmonic maps of radial-power form was first established by Ou (see \cite{Ou2014-f-biharmonic}, Proposition~2.9).
\end{remark}

\section{Liouville-type properties for proper $f$-polyharmonic functions and maps}

It is well known that on a closed (i.e., compact without boundary) Riemannian manifold, any $f$-harmonic function is constant. Indeed, the $f$-Laplace equation $f\Delta u + \langle\nabla f,\nabla u\rangle = 0$ can be rewritten as $\operatorname{div}(f\nabla u) = 0$. Multiplying this by $u$ and integrating over $M$, the divergence theorem yields
\[
0 = \int_M u \operatorname{div}(f\nabla u)\,dV = -\int_M f|\nabla u|^2\,dV.
\]
Since $f > 0$, it follows that $\nabla u \equiv 0$, hence $u$ is constant. For the case $k=2$, Ou \cite{Ou2014-f-biharmonic} proved that every $f$-biharmonic function on $M$ is also constant. In this section, we extend the above Liouville-type property to $f$-$k$-harmonic functions for arbitrary $k \ge 1$.

It is well known that every harmonic, subharmonic, or superharmonic function on a connected closed Riemannian manifold is constant. The following lemma demonstrates that this property carries over to the polyharmonic setting.

\begin{lemma}\label{lem:polyharmonic-constant}
Let $(M,g)$ be a connected closed Riemannian manifold and $u \in C^\infty(M)$.
If for some integer $p \ge 1$ we have
\[
\Delta^p u = 0,\quad\text{or}\quad \Delta^p u \ge 0,\quad\text{or}\quad \Delta^p u \le 0,
\]
then $u$ is constant.
\end{lemma}

\begin{proof}
We proceed by induction on $p$.
For $p = 1$, the statement follows from the classical maximum principle: every harmonic, subharmonic, or superharmonic function on a connected closed manifold is constant.

Assume the statement holds for $p-1$ with $p \ge 2$, and suppose $\Delta^p u = \Delta(\Delta^{p-1}u)$ satisfies one of the three conditions.
Then $\Delta^{p-1}u$ is respectively harmonic, subharmonic, or superharmonic.
Applying the case $p = 1$ forces $\Delta^{p-1}u$ to be constant, i.e., there exists a constant $C$ such that $\Delta^{p-1}u = C$.

Integrating this equality over $M$ and using the divergence theorem gives
\[
C \cdot \operatorname{Vol}(M) = \int_M \Delta^{p-1}u \, dV = 0,
\]
hence $C = 0$. Thus $\Delta^{p-1}u = 0$, and the induction hypothesis yields that $u$ is constant.
This completes the proof.
\end{proof}

Now, we can prove the following Liouville-type theorem for $f$-$k$-harmonic functions for arbitrary $k \ge 1$.
\begin{theorem}\label{thm:f-k-harmonic-constant}
Every $f$-$k$-harmonic function on a connected closed Riemannian manifold $M$ is constant.
\end{theorem}

\begin{proof}
We consider the even and odd cases separately.

\textbf{Case 1:} $k = 2s$ (even).
By \eqref{eq:f-k-harmonic}, an $f$-$k$-harmonic function $u$ satisfies
\[
\Delta^s(f \Delta^s u) = 0.
\]
Applying Lemma \ref{lem:polyharmonic-constant} to the function $f \Delta^s u$ shows that $f \Delta^s u$ is constant. We write
\[
f \Delta^s u = C_1
\]
for some constant $C_1$.
Since $f > 0$, we have $\Delta^s u = C_1/f$. Consequently, $\Delta^s u$ is either identically zero, strictly positive, or strictly negative on $M$.
A second application of Lemma \ref{lem:polyharmonic-constant} forces $u$ to be constant.

\textbf{Case 2:} $k = 2s+1$ (odd).
We now consider the case $s \ge 1$, as the case $s = 0$ has already been treated at the beginning of this section. In this case, an $f$-$k$-harmonic function $u$ satisfies
\[
\Delta^s\bigl(\operatorname{div}(f \nabla \Delta^s u)\bigr) = 0.
\]
Lemma \ref{lem:polyharmonic-constant} implies that $\operatorname{div}(f \nabla \Delta^s u)$ is constant, which we denote by $C_2$.
Integrating over $M$ and using the divergence theorem yields
\[
C_2 \cdot \operatorname{Vol}(M) = \int_M \operatorname{div}(f \nabla \Delta^s u) \, dV = 0,
\]
which forces $C_2 = 0$. Therefore
\[
\operatorname{div}(f \nabla \Delta^s u) = 0.
\]
Multiplying by $\Delta^s u$ and integrating by parts gives
\[
0 = \int_M \Delta^s u \cdot \operatorname{div}(f \nabla \Delta^s u) \, dV = -\int_M f \, |\nabla \Delta^s u|^2 \, dV.
\]
Since $f > 0$, we deduce $\Delta^s u = C_3$ for some constant $C_3$.
Integrating over $M$ again gives
\[
C_3 \cdot \operatorname{Vol}(M) = \int_M \Delta^s u \, dV = 0,
\]
hence $C_3 = 0$, and we obtain $\Delta^s u = 0$. Finally, Lemma \ref{lem:polyharmonic-constant} applied to $u$ yields that $u$ is constant.
\end{proof}

\begin{corollary}
Let $k \ge 1$ be an integer. Any $f$-$k$-harmonic map $\phi: M \to \mathbb{R}^n$ from a connected closed Riemannian manifold $M$ to Euclidean space is constant.
\end{corollary}
\begin{proof}
Write $\phi = (\phi^1,\dots,\phi^n)$. Each component $\phi^\alpha$ is an $f$-$k$-harmonic function by the flatness of $\mathbb{R}^n$. Theorem \ref{thm:f-k-harmonic-constant} then forces each $\phi^\alpha$ to be constant, hence $\phi$ is constant.
\end{proof}

\begin{corollary}\label{cor-u}
Let $u: \mathbb{R} \to \mathbb{R}$ be a smooth function and $f: \mathbb{R} \to \mathbb{R}^+$ a smooth positive function. Then $u$ is $f$-$k$-harmonic if and only if it can be expressed as
\begin{align}\label{u(x)}
  u(x)=&\underbrace{\int\cdots\int}_{k \text{ times}}\frac{C_{k-1}x^{k-1}
+\cdots +C_2x^2+C_1 x+ C_0}{f}\,(dx)^k \\[2mm]
   & + D_{k-1}x^{k-1}+\cdots +D_2x^2+D_1 x+ D_0, \nonumber
\end{align}
where $C_{k-1},\dots,C_0,D_{k-1},\dots,D_0$ are arbitrary constants.
\end{corollary}

\begin{proof}
By \eqref{eq:f-k-harmonic}, $u$ is $f$-$k$-harmonic iff $(f u^{(k)})^{(k)} = 0$.
Integrating $k$ times gives
\[
f(x)\, u^{(k)}(x) = C_{k-1}x^{k-1} + \cdots + C_1 x + C_0.
\]
Since $f(x)>0$, we have
\[
u^{(k)}(x) = \frac{C_{k-1}x^{k-1} + \cdots + C_0}{f(x)}.
\]
Integrating $k$ times yields exactly the expression in \eqref{u(x)}.

Conversely, substituting the form \eqref{u(x)} into $(f u^{(k)})^{(k)}$ clearly gives zero, so $u$ is $f$-$k$-harmonic.
\end{proof}
\begin{remark}
Corollary \ref{cor-u} generalizes Corollary 2.7 in \cite{Ou2014-f-biharmonic}. In particular, when $f(x)=e^{-\alpha x}$ for a constant $\alpha$, one easily checks that $u$ is proper $f$-$k$-harmonic if and only if it can be written as
\[
u(x)=P_{k-1}(x)e^{\alpha x}+Q_{k-1}(x),
\]
where $P_{k-1}$ and $Q_{k-1}$ are arbitrary polynomials of degree at most $k-1$. Hence, one obtains a large family of proper $f$-$k$-harmonic functions on the Euclidean line.
\end{remark}

\medskip\noindent
{\bf Data availability:} No new data were generated or analysed during this study.



\end{document}